\documentclass[smallextended]{svjour3}       
\smartqed  
\usepackage{mathptmx}      
\usepackage{graphicx}
\usepackage{amsmath,amssymb,latexsym}
\usepackage{algorithm,algorithmic}
\usepackage[labelformat=simple]{subfig}
\usepackage{bm}
\graphicspath{{./Figures/}}

\journalname{myjournal}

\newcommand{\set}[2]{\left\{{#1}\,:~{#2}\right\}}
\newcommand {\average}[1] {\mbox{$\left\{\!\!\left\{ #1 \right\}\!\!\right\}$}}
\newcommand {\jump}[1] {\mbox{$\left[\!\left[ #1 \right]\!\right]$}}

\begin{document}

\title{Energy preserving model order reduction of the nonlinear Schr\"odinger equation}


\titlerunning{Energy preserving MOR of NLSE}        

\author{  B\"ulent Karas\"ozen       \and
        Murat Uzunca
}


\institute{   B\"ulent Karas\"ozen \at
              Institute of Applied Mathematics \& Department of Mathematics \\
              Middle East Technical University, Ankara-Turkey \\
              Tel.: +90 312 2105602\\
              Fax: +90 312 2102985\\
              \email{bulent@metu.edu.tr}
           \and
           Murat Uzunca \at
           Department of Mathematics \\
           Sinop University, Sinop-Turkey\\
           Tel.: +90 368 271 5516-4285\\
           Fax: + 90 368 271 5520\\
                         \email{muzunca@sinop.edu.tr}
}

\date{Received: date / Accepted: date}

\maketitle


\begin{abstract}
An energy preserving reduced order model is developed for two dimensional nonlinear Schr\"odinger equation (NLSE) with plane wave solutions and with an external potential. The NLSE is discretized in space by the symmetric interior penalty discontinuous Galerkin (SIPG) method. The resulting system of Hamiltonian ordinary differential equations are integrated in time by the energy preserving average vector field (AVF) method.   The mass and energy preserving reduced order model (ROM) is constructed by proper orthogonal decomposition (POD) Galerkin projection.  The nonlinearities are computed  for the ROM efficiently by discrete empirical interpolation method (DEIM) and dynamic mode decomposition (DMD). Preservation of the semi-discrete energy and mass are shown for the full order model (FOM) and for the ROM which ensures the long term stability of the solutions. Numerical simulations illustrate the preservation of the energy and mass in the reduced order model for the two dimensional NLSE with and without the  external potential. The POD-DMD makes a remarkable improvement in computational speed-up over the POD-DEIM. Both methods approximate accurately the FOM, whereas POD-DEIM is more accurate than the POD-DMD.

\end{abstract}

\section{Introduction}
\label{intro}

The nonlinear Schr\"odinger equation (NLSE) is used frequently for modelling  wave propagation phenomena in different areas of  physics, chemistry and engineering.  In quantum physics and chemistry, the NLSE is used in the study of Bose-Einstein condensation (BEC), where it is also called  Gross-Pitaevskii equation (GPE). The numerical simulations are extremely  important for predicting the long term behavior of NLSE/GPE, because experiments to create BECs are challenging and computationally expensive.

The structure preserving symplectic and  multisymplectic \cite{Bridges06,Chen02,Islas01} integrators and energy preserving integrators  \cite{Celledoni12ped} using finite difference discretization in space, lead to long term stability of the solutions of Hamiltonian partial differential equations (PDEs) like the  NLSE. In \cite{Li15} local energy-preserving multi-symplectic algorithms in time are constructed for coupled and two dimensional  NLSE using pseudospectral methods
or Gauss-Legendre collocation methods for the spatial discretization.

 Due to the implicit nature of these integrators, at each time step, a coupled fully nonlinear system has to be solved by a fixed point or a Newton-Raphson method with high accuracy  \cite{Hairer10}. The discrete energy  is conserved when the nonlinear system is solved accurately  up to the machine precision. Therefore solving NLSE might be extremely time consuming, especially for two and three dimensional (2D and 3D) problems.

The goal of this work is fast and accurate reduced order modeling for the 2D  NLSE that preserves energy and mass. The proper orthogonal decomposition reduced order model (POD-ROM) has been widely used as a computationally efficient surrogate model in large-scale numerical simulations of complex systems. However, when it is applied to a Hamiltonian system like the NLSE, naive application of the POD method can destroy the Hamiltonian structure in the reduced order model. The stability of reduced models over long-time integration and the structure preserving properties has been recently investigated in the context of Lagrangian systems \cite{Lall03,Carlberg15}, and for Hamiltonian systems \cite{Beattie11,Chaturantabu16}. For linear and nonlinear Hamiltonian systems, a symplectic model reduction technique, proper symplectic decomposition (PSD) with symplectic Galerkin projection  is constructed in \cite{Peng16} which captures the symplectic structure of Hamiltonian systems to ensure long term  stability of the reduced model.  The PSD is computed using three algorithms; complex SVD, cotangent lift and nonlinear programming.  The PSD-DEIM method is extended with a greedy algorithm to parametric Hamiltonian systems in \cite{Hesthaven16}.  For nonlinear Hamiltonian systems, the PSD is combined with the discrete empirical interpolation method (DEIM)  to reduce the computational  cost \cite{Peng16,Hesthaven16}. Recently the energy preserving average vector field (AVF) method is used as a time integrator to construct reduced order models  for Hamiltonian systems \cite{Gong17}, for dissipative gradient systems like the Allen-Cahn equation \cite{Karasozen17}. In general, there exist no numerical integration method that preserves both the symplectic structure and the energy of a Hamiltonian systems.  The snapshots are enriched using the gradients of the Hamiltonian at time instances in \cite{Gong17}.  Besides the symplectic \cite{Hairer10} and multi-symplectic structure preserving geometric integrators \cite{Bridges06}, there are energy or Hamiltonian preserving integrators;  the average vector field (AVF) method \cite{Quispel08,Celledoni12ped}. The AVF is second order convergent in time, and preserves the time reversal symmetry of the Hamiltonian systems like the implicit mid-point rule.  Energy preservation of various PDEs by the AVF method is studied in  \cite{Celledoni12ped}: Hamiltonian and dissipative PDEs, like the NLSE, Korteweg-de  Vries equation, Sine-Gordon equation, Allen-Chan equation, Ginzburg Landau equation and Maxwell equation.

Currently the symplectic and energy preserving reduced order models are applied to one dimensional linear and non-linear Hamiltonian systems using finite differences in space; linear wave equation \cite{Peng16,Hesthaven16,Gong17}, NLSE  \cite{Peng16,Hesthaven16}, Sine-Gordon equation  \cite{Peng16}, Korteweg-de Vries equation \cite{Gong17}. The long term computation of the two dimensional NLSE with and without external potential is computationally challenging and more expensive than the one dimensional NLSE. In this paper  we  develop an energy stable reduced order model (ROM) for two dimensional NLSE, which inherits the energy preservation  property of the full order model (FOM).   For the space discretization we apply symmetric interior penalty discontinuous Galerkin (SIPG) method \cite{Arnold02,riviere08dgm}, and for time discretization the energy preserving AVF method. The SIPG is more efficient by evaluations of the nonlinear terms than the continuous FEM \cite{Karasozen15}. The AVF method preserves the energy dissipation of gradient systems like the Allen-Cahn equation \cite{Karasozen17a}  and the skew-gradient structure like the diffusive FitzHugh-Nagumo equation \cite{Karasozen15}. Reduced order models, which preserve the energy dissipation  of  gradient and skew-gradient systems are developed in  \cite{Karasozen15,Karasozen17}.  The POD-Galerkin approach  involves an offline-online splitting methodology. In the offline stage, the high fidelity or truth solutions are  generated by numerical simulations of the discretized high dimensional FOM. The POD is then  applied to compute an optimal subspace to fit the high fidelity data. A reduced system is constructed by projecting the high-dimensional system to this subspace. In the online stage, the reduced system is solved in the low-dimensional subspace. The primary challenge in producing the low dimensional models of the high dimensional discretized PDEs is the efficient evaluation of the nonlinearities (inner products) on the POD basis. Different POD Galerkin methods are developed to reduce the complexity of evaluating the nonlinear terms; gappy POD \cite{Everson95,Carlberg13}, missing point estimation \cite{Astrid08,Zimmermann16}, empirical and discrete empirical interpolation method (EIM/DEIM) \cite{Barrault04,chaturantabut10nmr}. Recently the dynamic mode decomposition (DMD) has appeared as an alternative to produce low rank approximation of the nonlinear terms  \cite{Alla16}. Like the DEIM, the DMD uses singular value decomposition (SVD) of the nonlinearities, which also uses sparse sampling through the gappy POD and DEIM/EIM. In \cite{Kutz13} a hybrid combination of the POD with DMD is applied to the one dimensional NLSE and to the Ginzburg-Landau equation.  Recently in \cite {Charnyi17}, for the incompressible Navier-Stokes equation (NSE), energy, mass and angular momentum conserving Galerkin discretization methods are developed in convective, skew-symmetric, conservative, and rotational formulations. The conservation properties of ROMs such as energy balance and mass conservation for the NLSE are investigated theoretically and numerically \cite{Mohebujjaman17}.

Efficient computation of the reduced order basis functions is an important aspect in reduced order modelling. The reduced order basis functions in POD-DEIM and POD-DMD are computed in the offline stage using singular value decomposition (SVD), which can be computationally demanding for large snapshot matrices. Here we use the randomized singular value decomposition (rSVD) \cite{Halko11a,Mahoney11,Martinsson16} as  a fast and accurate alternative to the deterministic SVD  to reduce the computational cost in the offline stage \cite{Alla16a,Bistrian17,Erichson16}. We compare the DEIM and DMD with respect to the accuracy and efficiency. We show that ROMs preserve well the energy and mass of the FOMs through numerical simulations for two dimensional NLSE with and without an external potential. It turns out that POD-DMD is several orders faster than the POD-DEIM, but less accurate than the POD-DEIM.

The remainder of the paper is organized as follows. The NLSE is reviewed briefly in Section~\ref{Sec:nlse}. Space-time discretization of NLSE by the SIPG and AVF method is given in Section~\ref{Sec:fom}. In Section~\ref{Sec:rom}, the energy preserving ROM, and also the POD-DEIM and POD-DMD are presented. The preservation of the invariants like the energy and mass of the NLSE, efficiency and accuracy of the ROMs are illustrated through the numerical simulations in Section~\ref{Sec:numeric}. The paper ends with some conclusions.

\section{Nonlinear Schr\"odinger equation}
\label{Sec:nlse}

The general form of the  NLSE equation with a cubic nonlinearity and an external potential $V$ is given  by

\begin{align}\label{p1}
 i \, \frac{\partial \psi(t,\bm{x})}{\partial t}  &=  \left [ - \alpha\Delta + \beta |\psi(t,\bm{x})|^2   +  V(\bm{x}) \right ]  \psi(t,\bm{x}), & (t,\bm{x})\in (0,T]\times \Omega,
\end{align}
with an initial condition

\begin{equation}
\psi(0,\bm{x} ) = \psi_0(\bm{x}), \quad \bm{x} \in \Omega,
\end{equation}
where $\psi(t,\bm{x})$ is the complex valued wave function,  $i$ is the complex unity ($i=\sqrt{-1}$), $\Delta = \nabla^2$ is the Laplace operator, $t$ is the time variable, $T>0$ is a  terminal time, $\alpha$ is a positive constant, and $\Omega \subset \mathbb{R}^d\; (d =1,2,3)$ is a bounded polygonal computational domain.  In this paper we consider two dimensional ($d=2$) NLSE, and we set the spatial variable $\mathbf{x} =(x,y)$. The parameter $\beta$ is a dimensionless constant with negative value for focusing (or attractive) and positive value for defocusing (or repulsive) nonlinearity. The external potential function $V(\mathbf{x})$ is a given real-valued function and its specific form depends on different applications \cite{Pitaevskii03,Sulem99}. For BEC, it is usually chosen as a harmonic confining potential, i.e., quadratic polynomial \cite{Antoine13,Bao13a}. The boundary conditions are usually of homogeneous Dirichlet or periodic type. We consider here periodic boundary conditions in the space time domain $(t,x,y) \in (0,T]\times [a,b] \times [c,d]$

\begin{align*}
\psi(t,a,y) &= \psi(t,b,y), \quad \forall (t,y) \in (0,T] \times [c,d],\\
\psi(t,x,c) &= \psi(t,x,d), \quad \forall (t,x) \in (0,T] \times [a,b].
\end{align*}

For the NLSE \eqref{p1}, we relate the mass (or density) functional $N(\psi(t,\bm{x}))$ and the energy functional $E(\psi(t,\bm{x}))$,  given as usual by

\begin{subequations}
\begin{align}
  N(\psi(t,\bm{x})) := &  \; \int_{\Omega}|\psi(t,\bm{x})|^2 \; d \bm{x},  \label{mass} \\
  E(\psi(t,\bm{x})) := & \; \frac{1}{2}\int_{\Omega} \left( \alpha |\nabla \psi(t,\bm{x})|^2 + V(\bm{x}) |\psi(t,\bm{x})|^2 + \frac{\beta}{2}|\psi(t,\bm{x})|^4 \right) \; d \bm{x}. \label{energy}
\end{align}
\end{subequations}
The NLSE \eqref{p1} is a Hamiltonian PDE which is well-known for Hamiltonian PDEs \cite{Sulem99,Bao13a} that it conserves the mass \eqref{mass} and the energy \eqref{energy} as the time progresses, in other words, the mass and energy at all times is equal to the mass and energy at the initial time $t=0$:
\begin{equation}
  N(\psi(t,\bm{x})) =  N(\psi_0(\bm{x})), \quad E(\psi(t,\bm{x})) =  E(\psi_0(\bm{x})), \quad \forall t\in (0,T],
\end{equation}
with the initial condition $\psi(0,{\mathbf x}) = \psi_0(\bm{x})$.

There are other important dynamical properties of the solution $\psi$ to the NLSE (\ref{p1}); time reversibility or symmetry, time transverse or gauge invariance, dispersion relation for NLSE without external potential. The Gross-Pitaevskii equation (GPE) is NLSE
for the macroscopic wave functions, governs the properties of BECs at temperatures far below the critical condensation temperature. The GPE includes the external potential term.

There exists a huge literature for the numerical solution of the NLSE, especially for 1D problems.
Numerical methods for 2D NLSE include the alternating direction implicit (ADI) method \cite{Xu12}, implicit-explicit multistep method \cite{gao16a}, high order compact finite difference schemes \cite{Wang13}, and local discontinuous Galerkin method \cite{Xu05}, the implicit Crank-Nicolson finite difference (CNFD), semi-implicit finite difference (SIFD), time splitting spectral and pseudo spectral methods \cite{Antoinegpelab1,Antoinegpelab2}.

\section{Full order model}
\label{Sec:fom}

In this section we introduce the space-time discretization of the NLSE \eqref{p1} leading to the so-called full order model (FOM), and we discuss the preservation of invariants for FOM under symmetric interior penalty Galerkin (SIPG) space discretization together with average vector field (AVF) time integrator.

Applying Green's formula to the NLSE \eqref{p1}, the continuous weak formulation of the problem (\ref{p1}) can be written as follows: for a.e. $t\in(0,T]$, find $\psi:=\psi(t,\bm{x}) \in W$ such that for any $\phi:=\phi(\bm{x}) \in W$

\begin{subequations}\label{p2}
\begin{eqnarray}
 i \, \langle \frac{\partial \psi}{\partial t}, \phi \rangle  & = &  \langle \alpha\nabla \psi, \nabla \phi \rangle  + \langle \beta |\psi|^2 \psi + V(\bm{x}) \psi, \phi \rangle,\\
 \langle \psi,\phi \big) &=& \langle \psi_0,\phi \rangle,
\end{eqnarray}
\end{subequations}
where $\langle \cdot , \cdot\rangle:= \langle\cdot , \cdot\rangle_{L^2(\Omega)}$ denotes the $L^2$-inner product on the space $L^2(\Omega)$ of square integrable functionals over a domain $\Omega\subset\mathbb{R}^2$, and together with the corresponding norm $\|\cdot\|_{L^2(\Omega)}$:

$$
L^2(\Omega) = \{f:\Omega \mapsto \mathbb{R} \; | \; \int_{\Omega}[f(\bm{x})]^2 d\bm{x} < \infty \},
$$
$$
\langle f(\bm{x}) , g(\bm{x})\rangle_{L^2(\Omega)} = \int_{\Omega}f(\bm{x})g(\bm{x})d\bm{x} \; , \quad \|f(\bm{x})\| := \|f(\bm{x})\|_{L^2}=\sqrt{\langle f(\bm{x}) , f(\bm{x})\rangle_{L^2(\Omega)}}.
$$

The solution space $W$ is set as the space of functions in the Sobolev space $H^{1}(\Omega)$ satisfying the periodicity on the boundary $\partial \Omega$ of the domain $\Omega$:

\begin{align*}
W &= \{\phi\in H^{1}(\Omega) | \; \phi \text{ periodic on }  \partial \Omega\},\\
H^{1}(\Omega) &= \{\phi\in L^{2}(\Omega) | \; \frac{\partial\phi}{\partial x}, \; \frac{\partial\phi}{\partial y}\in L^2(\Omega)\}.
\end{align*}

We decompose the complex valued wave function $\psi(t,\bm{x})$ into its real and imaginary parts as $\psi(t,\bm{x})=r(t,\bm{x}) +i s(t,\bm{x})$,  then   (\ref{p2}) can be re-written as a pair of real-valued system:

\begin{subequations}\label{p3}
\begin{align}
 \langle\frac{\partial r}{\partial t}, \eta \rangle   & =  \langle \alpha\nabla s, \nabla \eta \rangle   +  \langle\beta(r^2 +s^2)s + V(\bm{x})s , \eta \rangle, & \forall \eta \in W
 \\
 \langle\frac{\partial s}{\partial t}, \vartheta \rangle & =  - \langle \alpha\nabla r, \nabla \vartheta \rangle   -  \langle \beta(r^2 +s^2)r + V(\bm{x})r, \vartheta \rangle, & \forall  \vartheta \in W \\
 &  \langle r,\eta \rangle = \langle r_0,\eta \rangle, \qquad \langle s,\vartheta \rangle = \langle s_0,\vartheta \rangle, & \forall \eta , \vartheta \in W
\end{align}
\end{subequations}
where $\psi_0(\bm{x}) = r_0(\bm{x}) +i s_0(\bm{x})$. In the following, we introduce the SIPG method \cite{Arnold02,riviere08dgm} for the  spatial discretization of NLSE.  We remark that among the three common interior penalty Galerkin discretizations, nonsymmetric interior penalty Galerkin (NIPG), and incomplete interior penalty Galerkin (IIPG), only SIPG leads to a Hamiltonian system of ODEs \cite{Karasozen15nls}.

\subsection{Semi-discretization in space}

Consider for $h>0$, $\{ \mathcal{T}_h\}_h$ as a family of shape-regular simplicial triangulations of $\Omega$. Each mesh $\mathcal{T}_h$ consists of closed triangles such that $\overline{\Omega} = \bigcup_{K \in \mathcal{T}_h} \overline{K}$ holds. The diameter of an (triangular) element $K$ and the length of an edge $E$ are denoted by $h_{K}$ and $h_E$, respectively. We split the set of all edges $\mathcal{E}_h$ into the set $\mathcal{E}^{0}_{h}$ of interior edges and the set $\mathcal{E}^{p}_{h}$ of periodic boundary edge-pairs. An individual element of the set $\mathcal{E}^{p}_{h}$ is of the form $\omega =\{E_l, E_m\}$ where $E_l \subset \partial K_{n_l}\cap \partial\Omega$, and $E_m \subset \partial K_{n_m}\cap \partial\Omega$ is the corresponding periodic edge-pair of $E_l$ with $l>m$. Let the edge $E$ be a common edge for two elements $K$ and $K^e$. For a piecewise continuous scalar function $w$, there are two traces of $w$ along $E$, denoted by $w|_E$ from inside $K$ and $w^e|_E$ from inside $K^e$. Then, the jump and average of $w$ across the edge $E$ are defined by:

\begin{equation*}
\jump{w}=w|_E\bm{n}_{K}+w^e|_E\bm{n}_{K^e}, \quad
\average{w}=\frac{1}{2}\big( w|_E+w^e|_E \big),
\end{equation*}
where $\bm{n}_K$ (resp. $\bm{n}_{K^e}$) denotes the unit outward normal to $\partial K$ (resp. $\partial K^e$). Similarly, for a piecewise continuous vector field $\nabla w$, the jump and average across the edge $E$ are given by

\begin{equation*}
           \jump{\nabla w}=\nabla w|_E \cdot \bm{n}_{K}+\nabla w^e|_E \cdot \bm{n}_{K^e}, \quad
          \average{\nabla w}=\frac{1}{2}\big(\nabla w|_E+\nabla w^e|_E \big).
\end{equation*}
For a boundary edge $E \in K \cap \partial\Omega$, we set $\average{\nabla w}=\nabla w$ and $\jump{w}=w\bm{n}$, where $\bm{n}$ is the unit outward normal vector to $\partial\Omega$. The treatment of the periodic boundary edges is the same as the interior edges, in other words,  appropriate jump and average operators introduced in \cite{Vemaganti07} are used. Then, for each $\omega=\{ E_l, E_m\}\in \mathcal{E}^{p}_{h}$, we define the jump and average operators as follow

\[
\jump{w}_{\omega}=w|_{E_l}\bm{n}_l - w|_{E_m}\bm{n}_l, \quad \average{w}_{\omega}=\frac{1}{2} \big( w|_{E_l} + w|_{E_m} \big),
\]
where, for $l>m$, we use the fixed unit outward normal vector $\bm{n}_l$ to the boundary edge $E_l\subset\partial\Omega$.

For continuous finite elements methods (FEMs), the idea is to approximate the solutions $r$ and $s$ of (\ref{p3}) from a conforming, finite dimensional space $W_h \subset W$. On the other hand, we point out  that in discontinuous Galerkin (dG) methods the space of solutions or test functions consist of piecewise discontinuous polynomials. That is, no continuity constraints are explicitly imposed on the state and test functions across the element interfaces. As a consequence, weak formulations must include jump terms across interfaces, and typically penalty terms are added to control the jump terms. We define the space of solution and test functions by

\begin{align}\label{dg3}
W_h &=\set{w \in L^2(\Omega)}{ w\mid_{K}\in \mathbb{P}^q(K) \quad \forall K \in \mathcal{T}_h},
\end{align}
where $\mathbb{P}^q(K)$  is the set of polynomials  of degree at most $q$ in $K$. Note that the space of discrete solutions and the space of test functions  are identical due to the weak treatment of boundary conditions in dG methods. Note also that the space $W_h$ is a non-conforming space, i.e., $W_h \not\subset W$.

Let $r_h(0,\bm{x}), s_h(0,\bm{x})\in W_h$ be the projections (orthogonal $L^2$-projection) of the initial conditions $r_0(\bm{x})$ and $s_0(\bm{x})$ onto $W_h$, i.e., $\forall \eta(\bm{x}), \vartheta(\bm{x}) \in W_h$

\begin{equation}
\langle r_h(0,\bm{x}), \eta(\bm{x})\rangle = \langle r_0(\bm{x}),\eta(\bm{x})\rangle, \quad  \langle s_h(0,\bm{x}), \vartheta(\bm{x})\rangle = \langle s_0(\bm{x}),\vartheta(\bm{x})\rangle.
\end{equation}
Then, the weak formulation of the system (\ref{p3}), discretized by the SIPG method, reads as: for a.e. $t \in (0,T]$, find $r_h:=r_h(t,\bm{x})$, $s_h:=s_h(t,\bm{x}) \in W_h$ such that

\begin{subequations}\label{dg4}
\begin{align}
\langle \frac{\partial r_h}{\partial t}, \eta\rangle  & =   a_h(s_h, \eta) + b_{h,s}(r_h,s_h;\eta),  &  \forall \eta \in W_h, \label{dg4a} \\
\langle \frac{\partial s_h}{\partial t}, \vartheta\rangle & =  - a_h(r_h, \vartheta) - b_{h,r}(r_h,s_h;\vartheta), & \forall \vartheta \in W_h, \label{dg4b}
\end{align}
\end{subequations}
where the bi-linear form $a_h(\cdot , \cdot )$ and the non-linear forms $b_{h,z}(\cdot , \cdot ; \cdot )$ for $z\in\{r,s\}$ are given for any $\phi\in W_h$ by

\begin{subequations} \label{dg5}
\begin{align*}
a_h(z_h,\phi)=& \sum \limits_{K \in \mathcal{T}_h} \int \limits_{K}  \big(  \nabla z_h \cdot  \nabla \phi )\; d\bm{x}
                       -  \sum \limits_{ E \in \mathcal{E}^0_h } \int \limits_E \Big( \average{  \nabla z_h} \cdot \jump{\phi}   +  \average{  \nabla \phi} \cdot \jump{z_h} \Big) \; d\bm{s} \nonumber \\
                       & + \sum \limits_{K \in \mathcal{T}_h} \int \limits_{K}  V(\bm{x}) \, z_h \, \phi\; d\bm{x} - \sum_{\omega \in \mathcal{E}^{p}_{h}} \int \limits_{\omega} \Big( \average{   \nabla z_h}_{\omega} \cdot \jump{\phi}_{\omega} + \average{   \nabla \phi}_{\omega} \cdot \jump{z_h}_{\omega} \Big) \; d\bm{s}
                        \\
											&  + \sum \limits_{ E \in \mathcal{E}^0_h} \frac{\kappa \alpha }{h_E} \int \limits_E \jump{z_h} \cdot \jump{\phi} \; d\bm{s} + \sum_{\omega \in \mathcal{E}^{p}_{h}} \frac{\kappa \alpha }{h_{E}} \int \limits_{\omega} \jump{z_h}_{\omega} \cdot \jump{\phi}_{\omega} \; d\bm{s},\\
b_{h,z}(r_h,s_h;\phi)=&  \sum \limits_{K \in \mathcal{T}_h} \int \limits_{K} \beta \big( r_h^2 + s_h^2  \big)z_h \, \phi \; d\bm{x},
\end{align*}
\end{subequations}
where the parameter $\kappa$ in the above formulation is called penalty parameter which should be sufficiently large to ensure the stability of the SIPG discretization with a lower bound depending only on the polynomial degree $q$, see, e.g., \cite{riviere08dgm}.

The SIPG discretized semi-discrete solutions of \eqref{dg4} are given as

\begin{align}\label{dg6}
r_h(t,\bm{x}) = \sum \limits_{i=1}^{n_K} \sum \limits_{j=1}^{n_{q}} r_{j}^{i}(t) \varphi_{j}^{i}(\bm{x}), \qquad
s_h(t,\bm{x}) = \sum \limits_{i=1}^{n_K} \sum \limits_{j=1}^{n_{q}} s_{j}^{i}(t) \varphi_{j}^{i}(\bm{x}),
\end{align}
where $r_{j}^{i}(t), s_{j}^{i}(t)$, and $\varphi_{j}^{i}(\bm{x})$ are the unknown coefficients and the basis functions for $W_h$, respectively, for $j=1,2,\cdots, n_{q}$ and $i=1,2, \cdots, n_K$. The number $n_K$ denotes the number of (triangular) elements in $\mathcal{T}_h$, and $n_{q}$ is the local dimension on each element with the identity for the 2D problems $n_{q}= (q+1)(q+2)/2$, where $q$ is the degree of the polynomial order. Note that the degrees of freedom in dG methods are given by $N:=n_K\times n_q$, and throughout this paper we denote by $N$ the dimension of the high-fidelity model, i.e., the full order model (FOM).  Inserting the expansions (\ref{dg6}) into the system (\ref{dg4}), we obtain the following semi-discrete Hamiltonian system

\begin{equation}\label{fom}
\bm{M}\bm{z}_t = \bm{J}(\bm{A}\bm{z} + \bm{b}(\bm{z})),
\end{equation}
where ${\mathbf z}:={\mathbf z}(t) := ({\mathbf r}^T,{\mathbf s}^T)^T\in\mathbb{R}^{2N}$, and ${\mathbf r}:={\mathbf r}(t)\in\mathbb{R}^N$ and ${\mathbf s}:={\mathbf s }(t)\in\mathbb{R}^N$ are the unknown coefficient vectors for the solutions $r_h$ and $s_h$ with the ordered entries

\begin{align*}
{\mathbf r} &= (r_{1}^{1}(t), \cdots, r^{1}_{n_q}(t), \cdots, r_{n_K}^{1}(t), \cdots, r^{n_K}_{n_q}(t)),\\
{\mathbf s} &= (s_{1}^{1}(t), \cdots, s^{1}_{n_q}(t), \cdots, s_{n_K}^{1}(t), \cdots, s^{n_K}_{n_q}(t)).
\end{align*}
The other settings are given by

\begin{equation}\label{identities}
\bm{M} = \left[
  \begin{array}{cc}
    M & 0 \\
    0 & M\\
  \end{array}
\right], \;
\bm{A} = \left[
  \begin{array}{cc}
    A & 0 \\
    0 & A\\
  \end{array}
\right], \;
\bm{b}(\bm{z}) = \left[
  \begin{array}{cc}
    b({\mathbf r}, {\mathbf s};r_h) \\
    b({\mathbf r}, {\mathbf s}  ;s_h)
  \end{array}
\right], \; \bm{J} = \left[
  \begin{array}{cc}
    0 & \text{Id} \\
    -\text{Id} & 0\\
  \end{array} \right ],
\end{equation}
where $M\in\mathbb{R}^{N\times N}$ is the usual mass matrix, $A\in\mathbb{R}^{N\times N}$ is the stiffness matrix corresponding to the bilinear form $a_h(\cdot,\cdot)$, and $b(\cdot , \cdot ;z)\in\mathbb{R}^N$ are the vectors corresponding to the non-linear forms (inner products) $b_{h,z}(\cdot,\cdot;\cdot)$ for $z\in\{r,s\}$. The matrix ${\bm J}$ is the skew-symmetric matrix with $\text{Id}$ is the $N$ dimensional identity matrix.

\subsection{Time discretization}

For the temporal discretization of the semi-discrete system \eqref{fom}, we use the energy preserving, implicit, second order convergent AVF method \cite{Celledoni12ped,Quispel08}. For a system of  ordinary differential equations (ODEs)

\begin{equation*}
\dot{{\bm y}} = f({\bm y}),
\end{equation*}
the  AVF method  reads as

\begin{equation*}
\frac{{\bm y}_{n+1}-{\bm y}_{n}}{\tau} =  \int _{0}^{1} f (\xi {\bm y}_{n+1}+( 1-\xi ){\bm y}_{n}) \; d \xi,
\end{equation*}
where ${\bm y}_{n+1}$ is the unknown solution at time $t_{n+1}$, ${\bm y}_{n}$ is the known solution at time $t_{n}$, and $\tau =t_{n+1}-t_{n}$ is the uniform time-step size. The AVF method also preserves the energy of  Hamiltonian systems and Poisson systems with non-constant skew-symmetric structure \cite{Hairer10}, like the KdV equation and bi-Hamiltonian systems \cite{Karasozen13}. For problems with polynomial Hamiltonians, the AVF method can be interpreted as the Runge–Kutta method \cite{Celledoni14}, where the integrals can be evaluated exactly, and the implementation is comparable to that of the implicit mid-point rule. Using the multi-symplectic formulation of Hamiltonian PDEs, the AVF method is applied as an energy-preserving wavelet collocation to the NLSE and Camassa-Holm equation \cite{Gong16}, and as local energy-preserving method  to the NLSE and KdV equation \cite{Gong14}, to the 2D NLSE \cite{Li15}.

For time discretization, we divide the time interval $[0,T]$ into $N_T$ steps: $0=t_0 < t_1 < \cdots < t_{N_T}=T
$ with the uniform time step size $\tau = t_{n}-t_{n-1}$, $n=1,2,\cdots, N_T$.  We set $\bm{z}_{n} \approx \bm{z} (t_n)$ as the approximate solution vector at the time instance $t=t_n$, $n=0,1,\ldots ,N_T$, with $\bm{z}_n=[{\bm r}_n^T, {\bm s}_n^T]^T$, and ${\bm r}_{n} \approx {\bm r} (t_n)$ and ${\bm s}_{n} \approx {\bm s} (t_n)$. For $t=0$, let $r_h(0,\bm{x}), s_h(0,\bm{x})\in W_h$ be the projections (orthogonal $L^2$-projections) of the initial conditions $r_0({\bm x}), s_0({\bm x})$ onto $W_h$, and let ${\bm r}_0$ and ${\bm s}_0$ be the corresponding initial coefficient vectors satisfying the expansions (\ref{dg6}), and set $\bm{z}_0=[{\bm r}_0^T,{\bm s}_0^T]^T$. Then, applying the AVF method to the semi-discrete system \eqref{fom}, the full discrete problem of the NLSE \eqref{p1} reads as: for $n=0,1, \ldots , N_T-1$, find $\bm{z}_{n+1}$ satisfying

\begin{equation*}
\bm{M}(\bm{z}_{n+1}-\bm{z}_{n}) \; = \; \bm{J}\left( \frac{\tau}{2}\bm{A}(\bm{z}_{n+1}+\bm{z}_{n}) + \tau \int _{0}^{1} \bm{b}(\xi \bm{z}_{n+1}+( 1-\xi )\bm{z}_{n}) \; d \xi\right).
\end{equation*}

\subsection{Preservation of the invariants in FOM}

The NLSE  is time reversible or symmetric, i.e., the solutions are unchanged under the change of time as $t \rightarrow -t$ and taken conjugate in the equation \eqref{p1}. For the full discrete scheme, the time reversibility implies that the
scheme remains unchanged under the operation: $(n, n + 1) \longleftrightarrow (n + 1, n)$ and $\psi^n  \longleftrightarrow  \psi^{n+1}$ which is also true for the midpoint rule (Crank-Nicolson finite difference scheme) \cite{Antoine13} and SIPG-AVF scheme. There are two other invariants. The first one is the time transverse or gauge invariance, i.e., the translation of the potential $V\rightarrow V + \epsilon$ changes the phase of the solution as $\psi(t,\bm{x}) \rightarrow \psi(t,\bm{x}) e^{-i\epsilon t}$. The density or modulus of the solutions, $\rho=|\psi|^2$, is unchanged. NLSE admits plane wave solutions, satisfying a dispersion relation. NLSE without external potential (when $V(\bm{x})= 0$) generates plane wave solutions $\psi(t,\bm{x}) = A e^{i(\bm(k)\bm{x} \omega t)}$, where $A$ is the amplitude, $\omega$ is the time frequency, $\bm{k}$ is the spatial wave number. Like the Crank-Nicolson method, AVF method can not preserve the time translation invariant and the dispersion relation. For an overview of the preservation of invariants of NLSE at the discrete level and computational cost we refer to \cite{Antoine13}.

In the sequel, we show the conservation of discrete mass and discrete energy.

\begin{theorem}\label{thm_mass}
The semi-discrete solution $\psi_h(t,\bm{x})$ through the FOM \eqref{fom}, discretized by the SIPG method, conserves the discrete mass

\begin{equation}\label{est1}
N_h(\psi_h(t,\bm{x})) = \|\psi_h(t,\bm{x})\|^2 = \int_{\Omega} |\psi_h(t,\bm{x})|^2 \; d \bm{x} = \int_{\Omega} (r_h^2(t,\bm{x}) + s_h^2(t,\bm{x})) \; d \bm{x}
\end{equation}
exactly for all time.
\end{theorem}
\begin{proof}
To prove the discrete mass conservation, we utilize the weak formulation \eqref{dg4} of FOM. We choose the test functions $\eta=r_h$ in \eqref{dg4a} and $\vartheta=s_h$ in \eqref{dg4b}. Then, summation of \eqref{dg4a} and \eqref{dg4b}, and using the linearity of bilinear form $a_h$ in both arguments, yields

\begin{eqnarray*}
\langle \frac{\partial r_h}{\partial t}, r_h \rangle + \langle \frac{\partial s_h}{\partial t}, s_h \rangle =0,
\end{eqnarray*}
by which we obtain

\begin{equation*}
\frac{1}{2} \left( \frac{d}{dt} \|r_h\|^2 + \frac{d}{dt} \|s_h\|^2 \right) = \frac{1}{2} \left( \frac{d}{dt} \|r_h + i s_h \|^2 \right) \\
= \frac{1}{2}  \frac{d}{dt} \|\psi_h \|^2 =0.
\end{equation*}
From this equation we can see that the discrete mass $N_h(\psi_h(t,\bm{x}))$ is invariant for all time. \qed
\end{proof}

For the discrete energy preservation, we first introduce the following mesh-dependent energy functional which can be regarded as SIPG counterpart of the continuous energy functional $E(\cdot)$ defined in \eqref{energy}:

\begin{equation}\label{denergy}
E_h(r_h,s_h) := \Phi_h(r_h) + \Phi_h(s_h) + \frac{\beta}{4}\|r_h^2+s_h^2\|^2_{L^2(\mathcal{T}_h)},
\end{equation}
where

\begin{align*}
\Phi_h(w) = & \; \frac{1}{2}\langle \alpha\nabla w,\nabla w\rangle_{L^2(\mathcal{T}_h)} + \frac{1}{2}\langle V(\bm{x})w,w\rangle_{L^2(\mathcal{T}_h)} + J_h(w),  \\
J_h(w) := & \;  - \sum \limits_{ E \in \mathcal{E}^0_h } \int_E \average{w}\jump{w} d \bm{s} + \sum \limits_{ E \in \mathcal{E}^0_h } \frac{\kappa\alpha}{2h_E} \int_E \jump{w}\jump{w} d \bm{s} \\
 & \;  - \sum \limits_{ \omega \in \mathcal{E}^p_h } \int_{\omega} \average{w}_{\omega}\jump{w}_{\omega} d \bm{s} + \sum \limits_{ \omega \in \mathcal{E}^p_h } \frac{\kappa\alpha}{2h_E} \int_{\omega} \jump{w}_{\omega}\jump{w}_{\omega} d \bm{s}, \\
& \;  \|w\|^2_{L^2(\mathcal{T}_h)} = \langle w,w\rangle_{L^2(\mathcal{T}_h)} = \sum \limits_{K \in \mathcal{T}_h} \int_{K}w^2 d\bm{x}.
\end{align*}
Note that for a continuous solution $\psi = r+is$, since $J_h(r)=J_h(s)=0$, we have $E_h(r,s)=E(\psi)$.

\begin{theorem}
The FOM \eqref{fom} is a Hamiltonian ODE

\begin{equation}\label{fom_ham}
\bm{M}\bm{z}_t = \bm{J}\nabla_{\bm{z}} E_h(\bm{z})
\end{equation}
arising from the discrete energy functional $E_h(\bm{z})=E_h(r_h,s_h)$ defined in \eqref{denergy}, where $\bm{J}$ is the skew-symmetric matrix  and $\bm{z}=[\bm{r}^T,\bm{s}^T]^T$ is the vector of unknown coefficients $\bm{r}$ of $r_h$ and $\bm{s}$ of $s_h$. Moreover, for the discrete energy $E_h$  in \eqref{denergy}, the FOM \eqref{fom} satisfies the invariance of the energy for all time.
\end{theorem}
\begin{proof}
Firstly, we claim that $\nabla_{\bm{z}} E_h(\bm{z}) = \bm{A}\bm{z} + \bm{b}(\bm{z})$, for which then \eqref{fom_ham} is the FOM \eqref{fom}. For this, by the definition of the bilinear form $a_h$ in \eqref{dg5}, note that the  discrete energy functional $E_h(r_h,s_h)$ in \eqref{denergy} can be written as

\begin{equation}\label{bi_energy}
E_h(r_h,s_h) = \frac{1}{2}a_h(r_h,r_h) + \frac{1}{2}a_h(s_h,s_h) + \frac{\beta}{4}\|r_h^2+s_h^2\|^2_{L^2(\mathcal{T}_h)}.
\end{equation}
Then, Fr\'{e}chet derivatives of $E_h(r_h,s_h)$ with respect to the components $r_h$ and $s_h$ in any directions $\eta , \vartheta\in W_h$ are:

\begin{align*}
\frac{\delta E_h}{\delta r_h} \eta &= a_h(r_h,\eta) + \langle \beta (r_h^2+s_h^2)r_h , \eta \rangle_{L^2(\mathcal{T}_h)}, & \forall \eta\in W_h,\\
\frac{\delta E_h}{\delta s_h} \vartheta &= a_h(s_h,\vartheta) + \langle \beta (r_h^2+s_h^2)s_h , \vartheta \rangle_{L^2(\mathcal{T}_h)}, & \forall \vartheta\in W_h,
\end{align*}
or in matrix-vector form

\begin{align*}
\frac{\delta E_h}{\delta r_h} \eta &= A{\bm r} + b({\bm r},{\bm s};r_h),\\
\frac{\delta E_h}{\delta s_h} \vartheta &= A{\bm s} + b({\bm r},{\bm s}   ;s_h),
\end{align*}
Hence, with the definition of the  matrices and vectors in \eqref{identities}, we obtain that

\begin{equation}\label{eh1}
\nabla_{\bm{z}} E_h(\bm{z}) = \left[
\begin{array}{ll}
A{\bm r} + b( {\bm r},{\bm s}  ;r_h)\\
A{\bm s} + b({\bm r},{\bm s};s_h)
\end{array} \right] = \bm{A}\bm{z} + \bm{b}(\bm{z}).
\end{equation}
Further, using the identity \eqref{eh1} and the FOM \eqref{fom}, the invariance of the energy $E_h(\bm{z})$ for all time follows as:

\begin{equation*}
\frac{d}{dt}E_h(\bm{z}) = [ \nabla_{\bm{z}} E_h(\bm{z}) ]^T \bm{z}_t = [ \bm{A}\bm{z} + \bm{b}(\bm{z}) ]^T \bm{M}^{-1}\bm{J} [\bm{A}\bm{z} + \bm{b}(\bm{z})] = 0,
\end{equation*}
since the matrix $\tilde{\bm{J}}:=\bm{M}^{-1}\bm{J}$ is again a skew-symmetric matrix due to the fact that $\bm{J}$ is a skew-symmetric matrix and $\bm{M}$ is symmetric.\qed
\end{proof}

We also note that since the system of ODEs \eqref{fom} is Hamiltonian, the semi-discrete energy is the preserved by the AVF method \cite{Celledoni12ped}.

\section{Reduced order model}
\label{Sec:rom}

Because the computation of the FOM \eqref{fom} is time consuming,  in this section, we construct a small dimensional reduced order model (ROM) by utilizing the proper orthogonal decomposition (POD) method \cite{Kunisch01}. In addition to POD, the nonlinear vectors in the ROM are computed efficiently by discrete empirical interpolation method (DEIM) and dynamic mode decomposition (DMD). The low-rank approximation is computed in three steps: computation of the numerical solutions of the original high-dimensional system; dimensionality-reduction of the snapshot matrices by singular value decomposition (SVD);   Galerkin projection of the dynamics on the low-rank subspace. The first two steps are known as the offline stage, and the last one is the online stage. Offline stage is usually expensive and online step should be fast to run in real time.

\subsection{POD Galerkin projection}

For the $2N$-dimensional FOM \eqref{fom}, the ROM of lower dimension $k\ll 2N$  is formed by the Galerkin projection of the system onto a $k$-dimensional reduced space

$$
W_h^r=\text{span} \{ u_{1}, \ldots, u_{k} \}\subset [W_h]^2,
$$
resulting in the lower dimensional reduced solution $\psi_r(t,\bm{x})$ as:

\begin{equation}\label{rom_exp}
\psi_h(t,\bm{x})\approx \psi_h^r(t,\bm{x}) = \sum_{i=1}^{k} z_i^r(t) u_{i}(\bm{x}),
\end{equation}
where, without lost of generality, we have assumed that the complex solution $\psi_h=r_h+is_h$ is the column vector of the real and imaginary solutions $r_h$ and $s_h$, respectively, i.e., $\psi_h(t,\bm{x}):=[r_h(t,\bm{x}),s_h(t,\bm{x})]^T$, by which the solution $\bm{z}(t)$ of the FOM \eqref{fom} becomes the dG coefficient vector of $\psi_h(t,\bm{x})$ at time $t$. From the coefficients $z_i^r(t)$ in \eqref{rom_exp}, we set the solution of the reduced system as $\bm{z}^r(t):= \left( z_1^r(t),\ldots, z_k^r(t) \right)^T$. The functions $\{u_{i}(\bm{x}) \}_{i=1}^{k}$ in \eqref{rom_exp} are the orthogonal (in $L^2$-sense) reduced basis functions spanning the reduced space $W_h^r$. Belonging to the space $W_h^r\subset [W_h]^2$, the reduced  basis functions are linear combination of the dG  basis functions $\{ \varphi_j \}$, given by

\begin{equation*}
u_{i}(\bm{x}) = \sum_{j=1}^{2N} U_{j,i} \varphi_j(\bm{x}), \qquad i=1,\ldots ,k.
\end{equation*}
Then, using the column vectors $U_{i} = \left( U_{1,i}, \ldots ,U_{2N,i} \right)^T$, which are the coefficients of the $i$-th reduced basis function, we construct the following matrix of POD modes

\begin{equation}
\bm{U} := [ U_{1}, \ldots, U_{k}]\in \mathbb{R}^{2N\times k}.
\end{equation}

To obtain the reduced basis functions $\{u_{i}(\bm{x}) \}_{i=1}^{k}$, we need to solve the minimization problem \cite{Kunisch01}

\begin{align*}
\min_{u_{1},\ldots ,u_{k}} \frac{1}{N_T}\sum_{j=1}^{N_T} \left\| \psi_h(t_j,\bm{x}) - \sum_{i=1}^k (\psi_h(t_j,\bm{x}),u_{i})_{[L^2(\Omega)]^2}u_{i}\right\|_{[L^2(\Omega)]^2}^2 \\
\text{subject to } (u_{i},u_{j})_{[L^2(\Omega)]^2} = U_{i}^T\bm{M}U_{j}=\delta_{ij} \; , \; 1\leq i,j\leq k,
\end{align*}
where $\delta_{ij}$ is the Kronecker delta. The above minimization problem is equivalent to the eigenvalue problem \cite{Kunisch01}

\begin{equation}\label{eg1}
\mathcal{Z}\mathcal{Z}^T\bm{M}U_{i}=\sigma_{i}^2U_{i} \; , \quad i=1,2,\ldots ,k,
\end{equation}
where $\mathcal{Z}:= [\bm{z}_1 , \ldots, \bm{z}_{N_T}]\in\mathbb{R}^{2N\times N_T}$ is the snapshot matrix whose $n$-th column vector $\bm{z}_n$ is the solution vector of the FOM at time $t_n$. Then, the matrix $\bm{U}$ of POD modes can be computed through the singular value decomposition (SVD) of the snapshot matrix $\mathcal{Z}$. In addition, between the solution vector $\bm{z}$ of FOM and the reduced solution vector $\bm{z}^r$ of ROM, we have the relation

\begin{align*}
\bm{z} \approx \bm{U} \bm{z}^r, \quad \bm{z}^r \approx \bm{U}^T{\bm M} \bm{z},
\end{align*}
from where we can find the initial reduced vector $\bm{z}^r(0)$.  For a more detail description, we refer to \cite{Karasozen15}. We finally obtain the following reduced system:

\begin{equation}\label{rom}
  \frac{d}{dt} \bm{z}^r  =  \bm{A}^{\bm{r}} \bm{z}^r +  \bm{b}^{\bm{r}}( \bm{z}^r)
\end{equation}
with the reduced stiffness matrix and the reduced nonlinear vector

\begin{equation*}
\bm{A}^{\bm{r}} := \bm{U}^T\bm{J} \bm{A}, \quad \bm{b}^{\bm{r}}( \bm{z}^r) := \bm{U}^T\bm{J}\bm{b}( \bm{U}\bm{z}^r).
\end{equation*}
Like the FOM \eqref{fom}, the reduced system  (\ref{rom}) is   solved in time by the AVF method.

\subsection{Randomized singular value decomposition}

The SVD is known to be computationally demanding for large snapshot matrices resulting from the space-time discretized high dimensional model of the PDEs. Recently randomized algorithms are used in reduced order modelling to accelerate the offline computations in DMD \cite{Alla16a,Bistrian17,Erichson16}. Randomized methods for matrix computations provide an efficient computation of low-rank structures in data matrices, which are robust, reliable and computationally efficient and can be used to construct a smaller (compressed) matrix, which accurately approximates a high-dimensional data matrix.  The randomized singular value decomposition (rSVD) is robust, reliable and computationally efficient, which approximates the high dimensional snapshot matrices by constructing a smaller (compressed) matrix \cite{Halko11a,Mahoney11,Martinsson16}.

Given a snapshot matrix $Y\in {\mathbb R}^{m \times n}$ and the target rank $k\ll \min\{m,n\}$, a low rank approximation is constructed by creating a random matrix $\Omega
\in  {\mathbb R}^{n \times k}$, then the sampled matrix  $X= Y\Omega \in  {\mathbb R}^{m \times k}$. The entries of the random sampling matrix $\Omega$ are independent and identically distributed Gaussian random variables of zero mean and unit variance \cite{Alla16a}, and they are created by the MatLab routine \emph{randn}.  Afterward the $QR$ decomposition of $X$ is computed to obtain the orthonormal matrix $Q \in  {\mathbb R}^{m \times k}$, $X= QR$, so that

$$
Y \approx QQ^TY
$$
is satisfied. In the last step, the SVD of the small (compressed) matrix $B=Q^TY$ is computed. The approximation error of the rSVD can be decreased by introducing a small oversampling parameter $p$ (e.g. $p =2,5,10$). Instead of $k$  random vectors, $k+p$ are generated.

 \begin{algorithm}[htb!]
 \caption{Randomized SVD (rSVD)\label{rsvd} \cite{Alla16a} }
 \begin{algorithmic}
 \STATE Given the $m \times n$ matrix $Y$ and the target rank $k$
 \STATE Draw  $n \times k$ Gaussian random matrix  $\Omega$
 \STATE Compute the random matrix $ X = Y\Omega$
 \STATE Compute the QR decomposition $X = QR$
 \STATE Projection $B=Q^TY$
 \STATE Compute the deterministic SVD $B=\tilde{U}\Sigma V^T$
 \STATE Recover the right singular vectors $U=Q\tilde{U}$ of $Y$
 \end{algorithmic}
 \end{algorithm}

We use the rSVD for the SVD computations required in the POD reduced basis calculations, and the DEIM and DMD reduced basis calculations of the nonlinear terms.

\subsection{Preservation of the invariants in ROM}

In this section we show the preservation of discrete mass and discrete energy for the ROM \eqref{rom}. While the former is not difficult to show, we should state additional modifications to show the discrete energy preservation in ROM.

\begin{theorem}
The solution to the ROM \eqref{rom} conserves the discrete mass

\begin{equation*}
N_h(\psi_h^r(t,\bm{x})) = \|\psi_h^r(t,\bm{x})\|^2 = \int_{\Omega} |\psi_h^r(t,\bm{x})|^2 \; d \bm{x}
\end{equation*}
exactly for all time.
\end{theorem}
\begin{proof}
The proof of the mass conservation in ROM follows very similar to the proof for the FOM in Theorem~\ref{thm_mass}. But now, we consider the weak formulation not on the space $W_h$ of dimension $2N$ but on the reduced space $W_{h}^r\subset W_h$. \qed
\end{proof}

\begin{theorem}
The ROM \eqref{rom} with the discrete energy $E_h$ defined in \eqref{denergy}, does not necessarily satisfy the invariance of the discrete energy for all time. But the updated ROM
\begin{equation}\label{rom_pr}
\frac{d}{dt}\bm{z}^r = \bm{A}^r \bm{z}^r + \bm{b}^r( \bm{z}^r)
\end{equation}
with the updated reduced matrix $\bm{A}^r=\bm{J}_r\bm{U}^T\bm{A}$ and the updated reduced vector $\bm{b}^r=\bm{J}_r\bm{U}^T\bm{b}( \bm{U}\bm{z}^r)$, where the skew-symmetric matrix $\bm{J}_r$ is such that $\bm{U}^T\bm{J}= \bm{J}_r\bm{U}^T$, preserves the energy for all time.

\end{theorem}
\begin{proof}
Using the identity \eqref{eh1} for $\nabla_{\bm{z}} E_h(\bm{z})$ and the FOM \eqref{fom}, the reduced system \eqref{rom} can be written as

\begin{equation}\label{rom2}
\frac{d}{dt}\bm{z}^r = \bm{U}^T\bm{J}\nabla_{\bm{z}} E_h(\bm{U}\bm{z}^r).
\end{equation}
Then, for the discrete energy functional $E_h(\bm{U}\bm{z}^r)$ of the reduced system, we have that

\begin{align*}
\frac{d}{dt}E_h(\bm{U}\bm{z}^r) =& \; [ \nabla_{\bm{z}^r} E_h(\bm{U}\bm{z}^r) ]^T \frac{d}{dt}\bm{z}^r \\
=& \;  [ \bm{U}^T\nabla_{\bm{z}} E_h(\bm{U}\bm{z}^r) ]^T \bm{U}^T\bm{J} [\nabla_{\bm{z}} E_h(\bm{U}\bm{z}^r)]\\
=& \; [ \nabla_{\bm{z}} E_h(\bm{U}\bm{z}^r) ]^T \bm{U} \bm{U}^T\bm{J} [\nabla_{\bm{z}} E_h(\bm{U}\bm{z}^r)],
\end{align*}
which may be non-zero since the matrix $\bm{U} \bm{U}^T\bm{J}$ is not necessarily  a skew-symmetric matrix. But, if we can find a skew-symmetric matrix $\bm{J}_r$ so that $\bm{U}^T\bm{J}= \bm{J}_r\bm{U}^T$, then we can obtain an equivalent equation of the updated ROM \eqref{rom_pr} as

\begin{equation}\label{rom_pr2}
\frac{d}{dt}\bm{z}^r = \bm{J}_r[ \underbrace{\bm{U}^T\nabla_{\bm{z}} E_h(\bm{U}\bm{z}^r)}_{\nabla_{\bm{z}^r} E_h(\bm{U}\bm{z}^r)}].
\end{equation}
The equation \eqref{rom_pr2}, as a result the updated ROM \eqref{rom_pr}, is now Hamiltonian since $\bm{J}_r$ is skew-symmetric. Then, we also get for the discrete energy functional $E_h$ that

\begin{align*}
\frac{d}{dt}E_h(\bm{U}\bm{z}^r) =& \; [ \nabla_{\bm{z}^r} E_h(\bm{U}\bm{z}^r) ]^T \frac{d}{dt}\bm{z}^r \\
=& \;  [ \bm{U}^T\nabla_{\bm{z}} E_h(\bm{U}\bm{z}^r) ]^T \bm{J}_r[ \bm{U}^T \nabla_{\bm{z}} E_h(\bm{U}\bm{z}^r)]=0.
\end{align*}
Hence the ROM \eqref{rom_pr} satisfies the invariance of the discrete energy for all time. Indeed, the skew-symmetric matrix $\bm{J}_r$ can be easily computed as

$$
\bm{J}_r = \bm{U}^T\bm{J} (\bm{U}^T)^{-1} = \bm{U}^T\bm{J}\bm{M} \bm{U},
$$
where we used the $\bm{M}$-orthogonality of the POD modes $\bm{U}$, i.e., $\bm{U}^T\bm{M}\bm{U}=\bm{I}$.
\qed
\end{proof}

\begin{remark}
{\em The preservation of the reduced energy is valid for any reduced Hamiltonian system of ODEs  \eqref{rom_pr}. \/}

\end{remark}

\subsection{Approximation of the nonlinearities}

Although the reduced model \eqref{rom_pr} is of small dimension, the computation of the nonlinear vector $\bm{b}^r( \bm{z}^r)$ still depends on the dimension $2N$ of the FOM. In this section, we give two different approaches to reduce the computational complexity due to the nonlinear vector in the ROM \eqref{rom_pr}: discrete empirical interpolation method (DEIM) \cite{chaturantabut10nmr} and dynamic mode decomposition (DMD) \cite{Alla16a,Alla16}.

\subsubsection{Discrete empirical interpolation method (DEIM)}

The DEIM aims to find an approximation $\hat{\bm{b}}( \bm{U}\bm{z}^r)$ to the nonlinear vector $\bm{b}( \bm{U}\bm{z}^r)$, full dimensional nonlinear part of the reduced nonlinear vector $\bm{b}^r(\bm{z}^r)=\bm{J}_r\bm{U}^T\bm{b}( \bm{U}\bm{z}^r)$, by projecting it onto a subspace of the space generated by the non-linear vectors and spanned by a basis $\{ \bm{Q}_i\}_{i=1}^m$ of dimension $m\ll 2N$:

\begin{equation}\label{podG}
 \bm{b}( \bm{U}\bm{z}^r) \approx \hat{\bm{b}}( \bm{U}\bm{z}^r):=\bm{Q}a(t)
\end{equation}
where $\bm{Q}:= [\bm{Q}_1\; \ldots \; \bm{Q}_m]\in\mathbb{R}^{2N\times m}$ is the DEIM basis matrix, and $a(t)$ is the corresponding coefficient vector. Since the system \eqref{podG} is overdetermined, a projection matrix $\bm{P}=[e_{\mathfrak{p}_1},\ldots , e_{\mathfrak{p}_m}]\in\mathbb{R}^{2N\times m}$ with $e_{\mathfrak{p}_i}$ is the $i$-th column of the identity matrix $\bm{I}\in\mathbb{R}^{2N\times 2N}$ is computed. Then, the reduced model \eqref{rom_pr} can be rewritten as:

\begin{equation}\label{deim}
\frac{d}{dt} \bm{z}^r  =  \bm{A}^r \bm{z}^r +  \bm{J}_r\bm{B}\bm{b}^r_{\text{deim}}( \bm{z}^r)
\end{equation}
where the matrix $\bm{B}:= \bm{U}^T \bm{Q}(\bm{P}^T\bm{Q})^{-1}$ is computed once in the off-line stage (finite element discretization), and the reduced nonlinear vector $\bm{b}^r_{\text{deim}}( \bm{z}^r):= \bm{P}^T \bm{b}(\bm{U} \bm{z}^r)$ requires only $m\ll 2N$ integral evaluations. Using the definition of reduced matrix $\bm{A}^r=\bm{J}_r\bm{U}^T\bm{A}$, we also have an equivalent formulation of \eqref{deim} that

\begin{equation}\label{deim2}
\frac{d}{dt} \bm{z}^r  =  \bm{J}_r [ \bm{U}^T \bm{A}\bm{z}^r +  \bm{U}^T\hat{\bm{b}}( \bm{U}\bm{z}^r)],
\end{equation}
where $\hat{\bm{b}}( \bm{U}\bm{z}^r)= \bm{Q}(\bm{P}^T\bm{Q})^{-1}\bm{P}^T\bm{b}( \bm{U}\bm{z}^r)$ is the DEIM approximation of the nonlinear part $\bm{b}( \bm{U}\bm{z}^r)$. When DEIM approximation is not used, it requires $2N$ integral evaluations. On the other hand, the computation of the Jacobian of the nonlinear vector requires $2N\times n_p$ integral evaluations without DEIM, but it is only $m\times n_p$ with DEIM approximation.

For the details of the computation of the reduced non-linear vectors we refer to the greedy DEIM algorithm  \cite{chaturantabut10nmr}.  For continuous finite element and finite volume discretizations, the number of flops for the computation of bilinear form and nonlinear term depends on the maximum number of neighbor cells \cite{Drohmann12}. In the case of dG discretization, due to its local nature, it depends only on the number of nodes in the local cells. For instance, in the case of SIPG with linear elements ($n_q=3$), it contributes only 3 nonzero integrals for each degree of freedom, hence 3 integrals have to be computed  on a single triangular element \cite{Karasozen15}, whereas in the case of continuous finite elements, integral computations on $6$ neighbor cells are needed \cite{Heinkenschloss14}. Since the AVF method is an implicit time integrator, at each time step, a non-linear system of equations has to be solved by Newton's method. The reduced Jacobian has a block diagonal structure for the SIPG discretization, which is easily invertible \cite{Karasozen15}, and requires only $O(n_qN)$ operations with DEIM.

Because the nonlinearity in the ROM \eqref{rom_pr} is approximated  by DEIM, the discrete energy
is also preserved approximately in \eqref{deim}, not exactly as for the ROM \eqref{rom_pr}, since we have that

\begin{align*}
\frac{d}{dt}E_h(\bm{U}\bm{z}^r) =& \; [ \nabla_{\bm{z}^r} E_h(\bm{U}\bm{z}^r) ]^T \frac{d}{dt}\bm{z}^r \\
=& \;  [ \bm{U}^T\nabla_{\bm{z}} E_h(\bm{U}\bm{z}^r) ]^T \bm{J}_r[ \bm{U}^T \bm{A}\bm{z}^r +  \bm{U}^T\hat{\bm{b}}( \bm{U}\bm{z}^r)]\\
\neq& \; [ \bm{U}^T\nabla_{\bm{z}} E_h(\bm{U}\bm{z}^r) ]^T \bm{J}_r[ \bm{U}^T\nabla_{\bm{z}} E_h(\bm{U}\bm{z}^r)].
\end{align*}
But an upper  bound for the preservation of the discrete energy by POD-DEIM can be derived as follows:

\begin{align*}
\frac{d}{dt}E_h(\bm{U}\bm{z}^r) =& \; [ \nabla_{\bm{z}^r} E_h(\bm{U}\bm{z}^r) ]^T \frac{d}{dt}\bm{z}^r \\
=& \;  [ \bm{U}^T\nabla_{\bm{z}} E_h(\bm{U}\bm{z}^r) ]^T \bm{J}_r[ \bm{U}^T\bm{A}\bm{z}^r +  \bm{U}^T\hat{\bm{b}}( \bm{U}\bm{z}^r)] \\
=& \;  [ \bm{U}^T\nabla_{\bm{z}} E_h(\bm{U}\bm{z}^r) ]^T \bm{J}_r[ \bm{U}^T\bm{A}\bm{z}^r +  \bm{U}^T\hat{\bm{b}}( \bm{U}\bm{z}^r) - \bm{U}^T\nabla_{\bm{z}} E_h(\bm{U}\bm{z}^r)],
\end{align*}
where in the last row, we have added $- \bm{U}^T\nabla_{\bm{z}} E_h(\bm{U}\bm{z}^r)$ which does not affect the identity since $[ \bm{U}^T\nabla_{\bm{z}} E_h(\bm{U}\bm{z}^r) ]^T\bm{J}_r[ \bm{U}^T\nabla_{\bm{z}} E_h(\bm{U}\bm{z}^r) ]=0$. Using that $\bm{U}^T\nabla_{\bm{z}} E_h(\bm{U}\bm{z}^r) = \bm{U}^T \bm{A}\bm{z}^r +  \bm{U}^T\bm{b}( \bm{U}\bm{z}^r)$, we obtain

\begin{equation*}
\frac{d}{dt}E_h(\bm{U}\bm{z}^r) =  [ \bm{U}^T\nabla_{\bm{z}} E_h(\bm{U}\bm{z}^r) ]^T \bm{J}_r\bm{U}^T[ \hat{\bm{b}}( \bm{U}\bm{z}^r) - \bm{b}( \bm{U}\bm{z}^r)],
\end{equation*}

\begin{equation}\label{bound_deim}
\left\| \frac{d}{dt}E_h(\bm{U}\bm{z}^r) \right\| \leq \|\nabla_{\bm{z}} E_h(\bm{U}\bm{z}^r)\| \|\bm{U}\|^2 \|\bm{J}_r\| \|\hat{\bm{b}}(\bm{U}\bm{z}^r) - \bm{b}( \bm{U}\bm{z}^r)\|.
\end{equation}
Finally, using the DEIM approximation error \cite[Lemma 3.2]{chaturantabut10nmr}

$$
\|\hat{\bm{b}}(\bm{U}\bm{z}^r) - \bm{b}( \bm{U}\bm{z}^r)\| \leq \|(\bm{P}^T\bm{Q})^{-1}\| \|(\bm{I}-\bm{Q}\bm{Q}^T)\bm{b}( \bm{U}\bm{z}^r)\|,
$$
we obtain an upper bound for the change of the energy with respect to time:

\begin{equation}\label{bound_deim2}
\left\| \frac{d}{dt}E_h(\bm{U}\bm{z}^r) \right\| \leq \|\nabla_{\bm{z}} E_h(\bm{U}\bm{z}^r)\| \|\bm{U}\|^2 \|\bm{J}_r\| \|(\bm{P}^T\bm{Q})^{-1}\| \|(\bm{I}-\bm{Q}\bm{Q}^T)\bm{b}( \bm{U}\bm{z}^r)\|.
\end{equation}
According to the bound \eqref{bound_deim2}, the discrete energy is preserved since $\|\frac{d}{dt}E_h\|\rightarrow 0$ as $\|(\bm{I}-\bm{Q}\bm{Q}^T)\bm{b}( \bm{U}\bm{z}^r)\| \rightarrow 0$, and that $\bm{Q}$ is orthonormal. We remark  that similar error bounds for the reduced order energy are derived for the symplectic DEIM in \cite{Hesthaven16,Peng16}.

\subsubsection{Dynamic mode decomposition (DMD)}

The DMD extracts dynamically relevant spatio-temporal information content from a numerical or experimental data sets \cite{Kutz16}. It is a powerful equation-free, data-driven method to analyze complex systems. Without explicit knowledge of the dynamical system, the DMD algorithm determines eigenvalues, eigenmodes, and spatial structures for each mode. The order of orthogonal POD modes are  decreasing  by the energy level of POD singular values. The DMD modes are not orthogonal and each DMD mode is associated with a growth rate and each mode has a single frequency. Again the DMD modes are ordered in form of decreasing energy; unlike the POD modes, they give the energy fluctuations at different frequencies.  After building the DMD basis functions of rank $m$,  we approximate the full dimensional nonlinear part $\bm{b}( \bm{U}\bm{z}^r)$ of the nonlinear term $\bm{b}^r=\bm{J}_r\bm{U}^T\bm{b}( \bm{U}\bm{z}^r)$ in \eqref{rom_pr} following  \cite{Alla16a,Alla16}.

DMD is a special case of the Koopman operator  \cite{Koopman31} approximating nonlinear systems via an associated infinite dimensional system. The connection between the DMD and Koopman operator was  established  in \cite{Mezic13,Rowley12,Schmid10dmd} . The Koopman operator ${\mathcal K}$ acts on a set of scalar observable functions $g:{\mathcal M} \rightarrow {\mathbb C}$

$$
{\mathcal K} g({\bm y}) = g({\bm N}({\bm y})) ,
$$
for the nonlinear dynamical system

\begin{equation}\label{nonlin}
\frac{d{\bm y}}{dt} = {\bm N}({\bm y}),
\end{equation}
where ${\bm y}\in {\mathcal M}$, an n-dimensional manifold. The DMD determines the Koopman eigenvalues and  modes directly from the data, when the observable is considered as state space, $g({\bm y}) = {\bm y}$. For construction of the DMD modes, we follow (Chapter 1,\cite{Kutz16} ).  The nonlinear system \eqref{nonlin} is approximated locally by the following linear system

\begin{equation}
\frac{d{\bm y}}{dt} = \mathcal{A}{\bm y},
\end{equation}
 by constructing the discrete dynamical system

$$
\bm{y}_{i + 1} = \bm{A} \bm{y}_i, \quad i = 0,1, \dots, J,
$$
for $J+1$ trajectories, and $ \bm{A} = exp(\mathcal{A}\Delta t)$. Low rank eigen-decomposition of the matrix $\bm{A}$ is constructed  by minimizing

\begin{equation} \label{minimization}
\| \bm{y}_{i + 1} - \bm{A} \bm{y}_i \|_2,
\end{equation}
in the least squares sense for all trajectories  $\bm{y}_i,\;  i = 0,1, \dots, J$. For minimization of the approximation error \eqref{minimization},  we consider the  snapshot matrices ${\bm G}$ and ${\bm G}'$ formed as:

$$
{\bm G} = [ \bm{y}_0, \cdots, \bm{y}_{J-1} ], \qquad   {\bm G}'   = [ \bm{y}_1, \cdots, \bm{y}_{J} ].
$$
Then,  we find the unknown matrix ${\bm A}_G$ satisfying ${\bm G }' = {\bm A}_G   {\bm G }$, which is the solution of the minimization problem in the Frobenius norm

\begin{align*}
\min   \left\| {\bm G }' -  {\bm A}_G   {\bm G }\right\|_{F}^2,
\end{align*}
so that ${\bm A}_G = {\bm G}'  {\bm G}^{\dag}$, where $\dag$ denotes the Moore-Penrose pseudo inverse. The DMD modes are computed by the exact DMD algorithm \cite{Tu14}.

 \begin{algorithm}[htb!]
 \caption{Exact DMD Algorithm \label{exactDMD_alg}}
 \begin{algorithmic}
 \STATE Given the snapshot matrices ${\bm G}$ and ${\bm G}'$
 \STATE Compute rSVD (Algorithm \ref{rsvd}) of ${\bm G}$,  ${\bm G} = U\Sigma V^{*}$.
 \STATE Define $\tilde{\bm A}_G = U^{*} {\bm G}' V \Sigma^{-1}$.
 \STATE Find eigenvalues and eigenvectors of $\tilde{\bm A}_G W=  W \Lambda$.
 \STATE Set DMD modes $U^{\text{DMD}}:={\bm G}'V\Sigma^{-1}W$.
 \end{algorithmic}
 \end{algorithm}

In our case, the exact DMD algorithm, Algorithm \ref{exactDMD_alg},  is applied to the  snapshot matrices ${\bm G}$ and ${\bm G}'$ formed by the snapshots of the nonlinear terms
${\bm b}(\bm{z})$ in the FOM \eqref{fom} at $N_T+1$ equally spaced time instances as:

$$
{\bm G} = [ {\bm b}(\bm{z}_0), \cdots, {\bm b}(\bm{z}_{N_T-1}) ], \qquad   {\bm G}'   = [ {\bm b}(\bm{z}_1), \cdots, {\bm b}(\bm{z}_{N_T}) ].
$$
Then, we obtain the time dependent DMD approximation to the nonlinear part ${\bm b}(\bm{U}\bm{z}^r(t))$ in the ROM \eqref{rom_pr} as

\begin{equation}
 {\bm b}(\bm{U}\bm{z}^r(t)) \approx \bm{b}_{\text{dmd}}^r(t) = \sum_{j=1}^{m}\alpha_jU_j^{\text{DMD}}(z) \exp(\omega_ j t) = U^{\text{DMD}}\text{diag}(e^{\omega^{\text{DMD}} t})\alpha ,
 \end{equation}
 where $ U^{\text{DMD}}=[U_1,\ldots,U_m ]$ are DMD basis functions of rank $m$ of the nonlinear vector ${\bm b}(\bm{z})$,  $\alpha=[\alpha_1,\ldots, \alpha_m]$  is the initial vector $\alpha = (U^{\text{DMD}})^{\dag}{\bm b}(\bm{z}_0)$  and   $\omega_j=\log{(\lambda_j)}/\Delta t$, $j=1,\ldots m$. After plugging this term into \eqref{rom_pr}, we obtain the following linear ROM:

\begin{equation}\label{dmd}
 \frac{d}{dt} \bm{z}^r  =  \bm{A}^r \bm{z}^r +  \bm{J}_r\bm{U}^T\bm{b}_{\text{dmd}}^r(t).
\end{equation}

The POD-DMD ROM \eqref{dmd} corresponds to the semi-discretized linear Schr\"odinger equation with an additional time dependent term coming from the linearization by POD-DMD approximation. The reduced energy of the POD-DMD  contains only quadratic terms in (3b) like linear Schr\"odinger equation  \cite{Faou09}. The AVF method is equivalent to the midpoint rule for quadratic potentials which is also energy preserving \cite{Celledoni12ped,Cohen11}. Because the POD-DMD approximation of the cubic nonlinearity of the NLSE is reduced to a time dependent function appearing as additional term in the linear  Schr\"odinger equation, the mass and energy are preserved with a lower accuracy than  by POD-DEIM, which is shown by the numerical results in the next section. Although the dimension of the system \eqref{dmd} is the same as \eqref{deim}, the main advantage of the system \eqref{dmd} is that it is linear and we do not need to use the Newton's method. Therefore the POD-DMD is significantly much faster than POD and POD-DEIM.

\section{Numerical results}
\label{Sec:numeric}

In this section we present numerical results for the NLSE \eqref{p1} on a 2D localized rectangle  $\Omega = [a,b]^2$ with periodic boundary conditions. In all simulations we use linear dG basis functions on a uniform $32\times 32$ rectangular grid with $2048$ triangular elements. Because the NLSE generate wave type solutions, it is not possible to capture the dynamics with few modes. For numerical simulations, the number of POD modes are fixed according to the relative information content or energy criterion $\varepsilon_k$ given by

\begin{equation}\label{ric}
\varepsilon_k = \frac{\sum_{i=1}^k \sigma_i^2}{\sum_{i=1}^{d_z} \sigma_i^2},
\end{equation}
where $\sigma_i$ is the $i$-th singular value in the SVD of the snapshot matrix $\mathcal{Z}$ for computing the matrix $\bm{U}$ of POD modes, and $d_z$ is the rank of the snapshot matrix $\mathcal{Z}$. We set the number of POD modes as $\min_{k}\varepsilon_k>0.9999$ which sufficiently reflects the system characteristics. In the POD and DEIM/DMD modes computation procedure, we use rSVD algorithm, Algorithm \ref{rsvd}, for the computation of SVD with the random matrix $\Omega$ and  the oversampling parameter $p=2$ . In all numerical experiments we have used for the FOM and ROMs the same time-step sizes. All simulations are performed on a Windows 10 machine with Intel Core i7, 2.5 GHz and 8 GB using MATLAB R2014.

\subsection{Defocusing NLSE with progressive wave solutions}
\label{ex1}

We first consider defocusing NLSE \eqref{p1} ($\beta =2$)  with the progressive plane wave solution \cite{Wang13})

\[
\psi (t,\mathbf{x}) = A \exp\big( i (c_1 x + c_2 y - \omega t) \big),
\]
where $\omega=c_1^2 + c_2^2 - \beta |A|^2$ and $\alpha = 2$. The initial data is evaluated by taking $t =0$ using the exact solution. Numerical solutions are obtained with the linear dG elements in the spatial domain $\Omega=[0,2 \pi]^2$ and for the final time $T=5$, with the spatial and temporal mesh sizes $h=\pi/16$ and $\tau =0.001$, respectively. The parameter values are taken as $A=1$, $c_1=1$ and $c_2=1$.

\begin{figure}[htb!]
\centering
\subfloat{\includegraphics[scale=0.35]{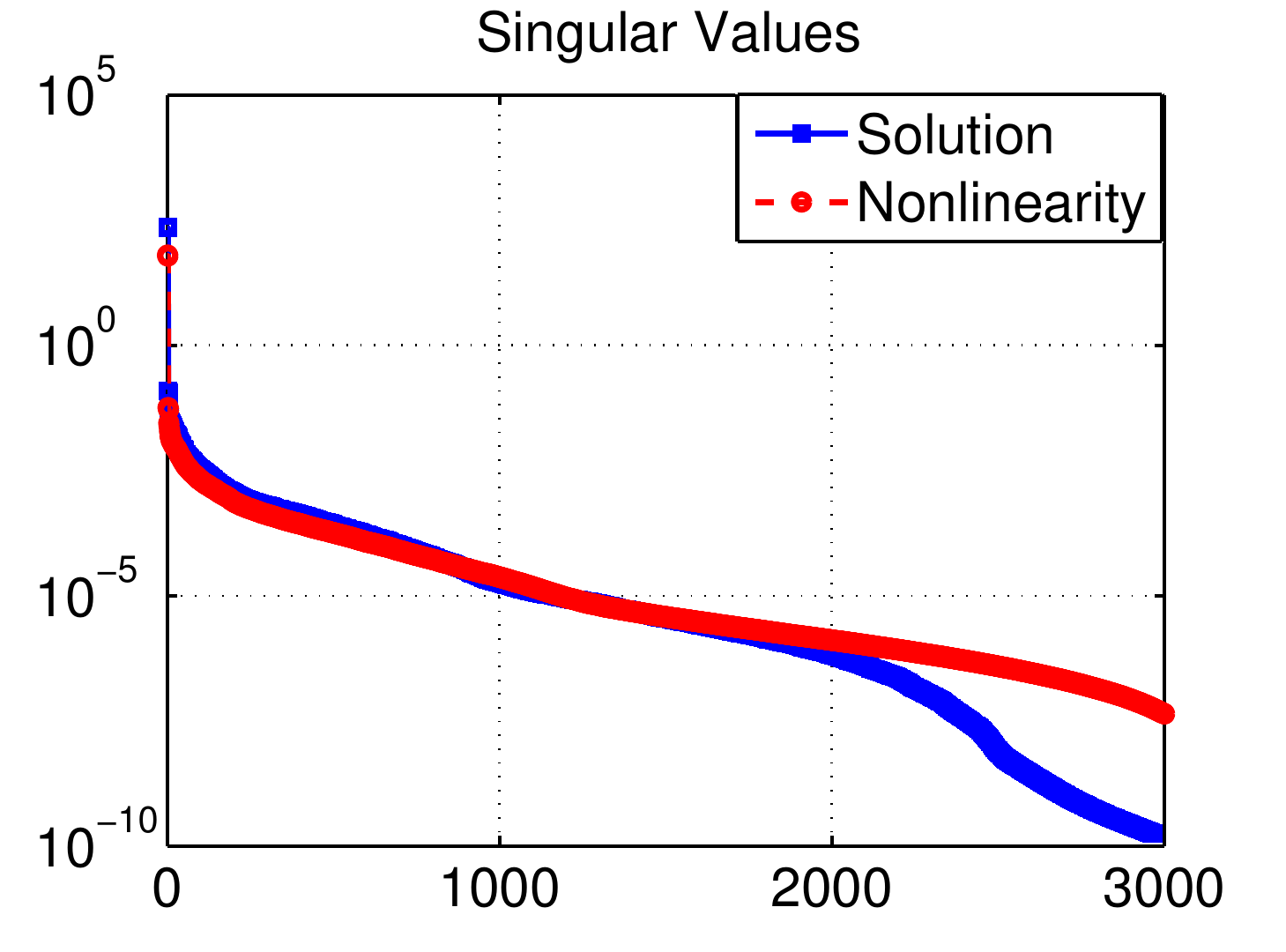}}
\caption{Example~\ref{ex1}: Decay of singular values\label{ex1_svd}}
\end{figure}

In Fig.~\ref{ex1_svd} the decay of singular values are shown for the snapshot matrices of the solutions and nonlinear terms. At the very beginning a steep decay of singular values is observed. With increasing number of the modes, the singular values decay rather slowly. We choose in the ROM the number of POD modes $k=10$ which satisfies the energy criterion \eqref{ric} with  $\varepsilon_{10}>0.9999$. In Fig.~\ref{ex1_plots} the  FOM solution and the errors between FOM and ROMs are shown at the final time. The number of DEIM and DMD modes are chosen to obtain the same level of accuracy for the error  between the FOM and ROM solutions \cite{Alla16}. Both POD-DEIM and POD-DMD approximations with 15 DEIM/DMD modes for 10-rank truncation (POD modes) have almost the same accuracy.

\begin{figure}[htb!]
\centering
\subfloat{\includegraphics[scale=0.3]{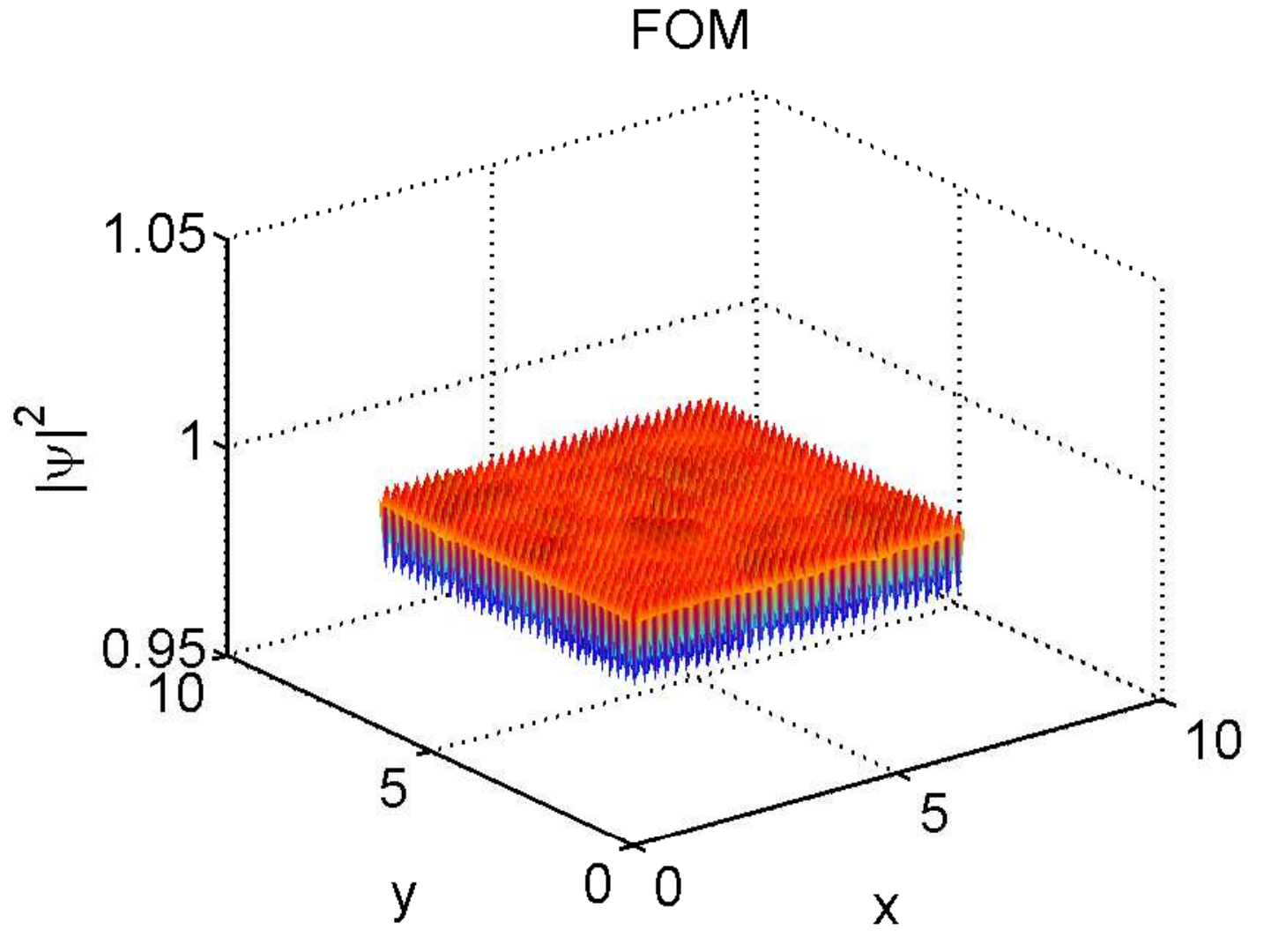}}
\subfloat{\includegraphics[scale=0.3]{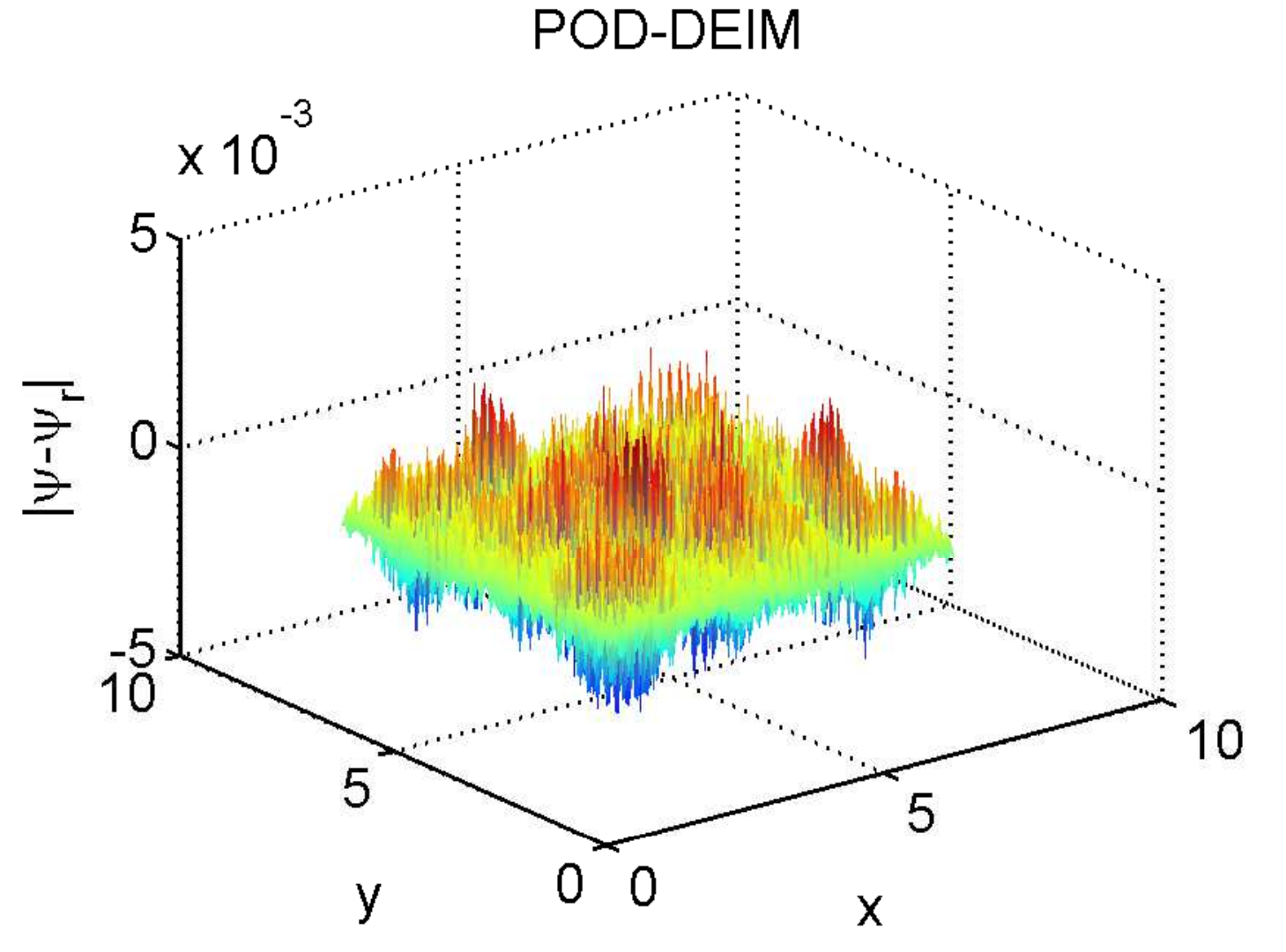}}
\subfloat{\includegraphics[scale=0.3]{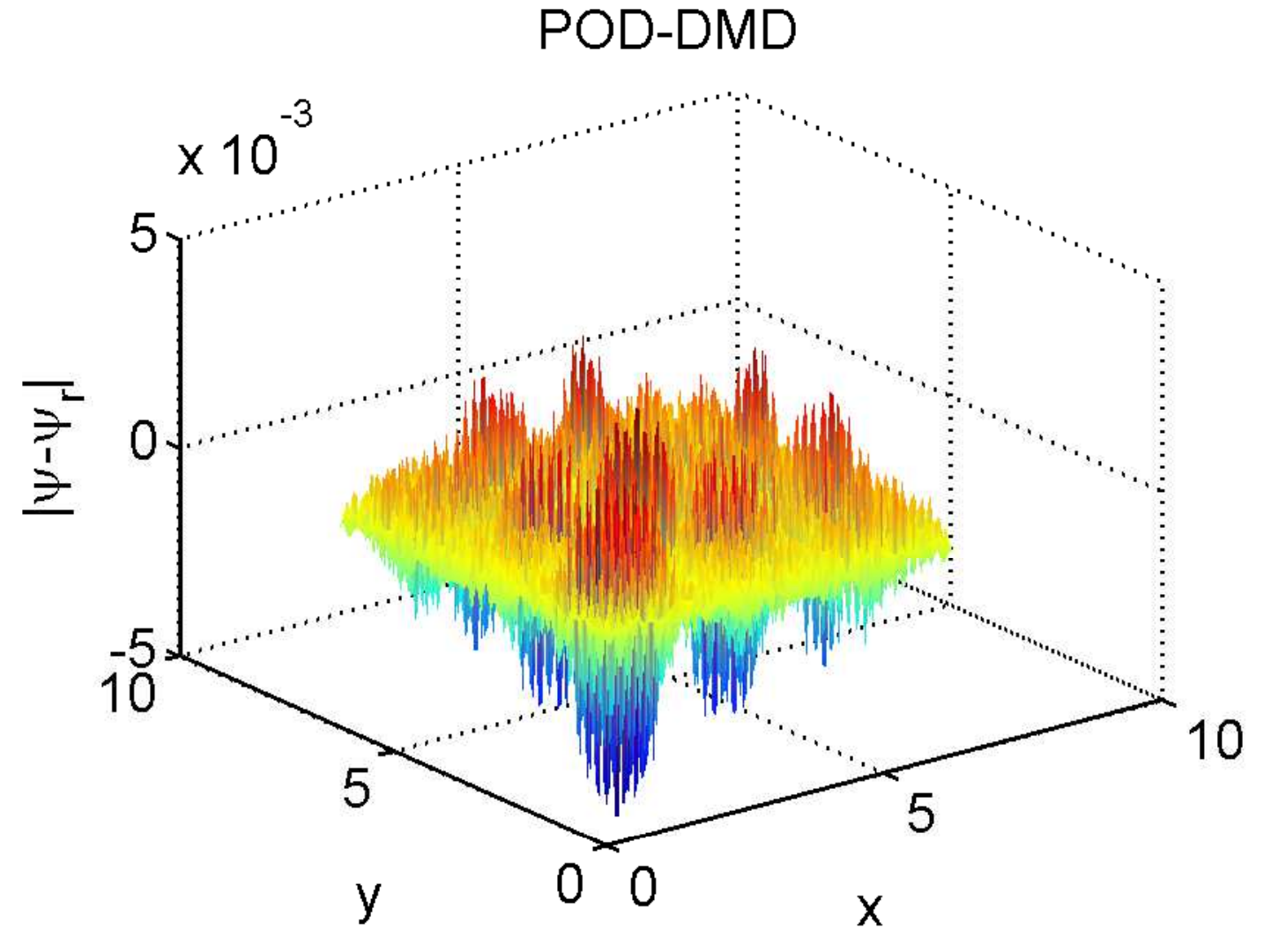}}
\caption{Example~\ref{ex1}: Solution profile by FOM (left), the errors between FOM and ROMs with 10 POD 15 DEIM (middle) and 10 POD 15 DMD (right) modes, at the final time $T=5$\label{ex1_plots}}
\end{figure}

In Fig.~\ref{ex1_energymass}, the energy and mass errors are shown for the FOM and the ROM in $L^{\infty}$-norm. The ROM errors are larger than the FOM errors as expected, but both mass and energy are well preserved.

\begin{figure}[htb!]
\centering
\subfloat{\includegraphics[scale=0.3]{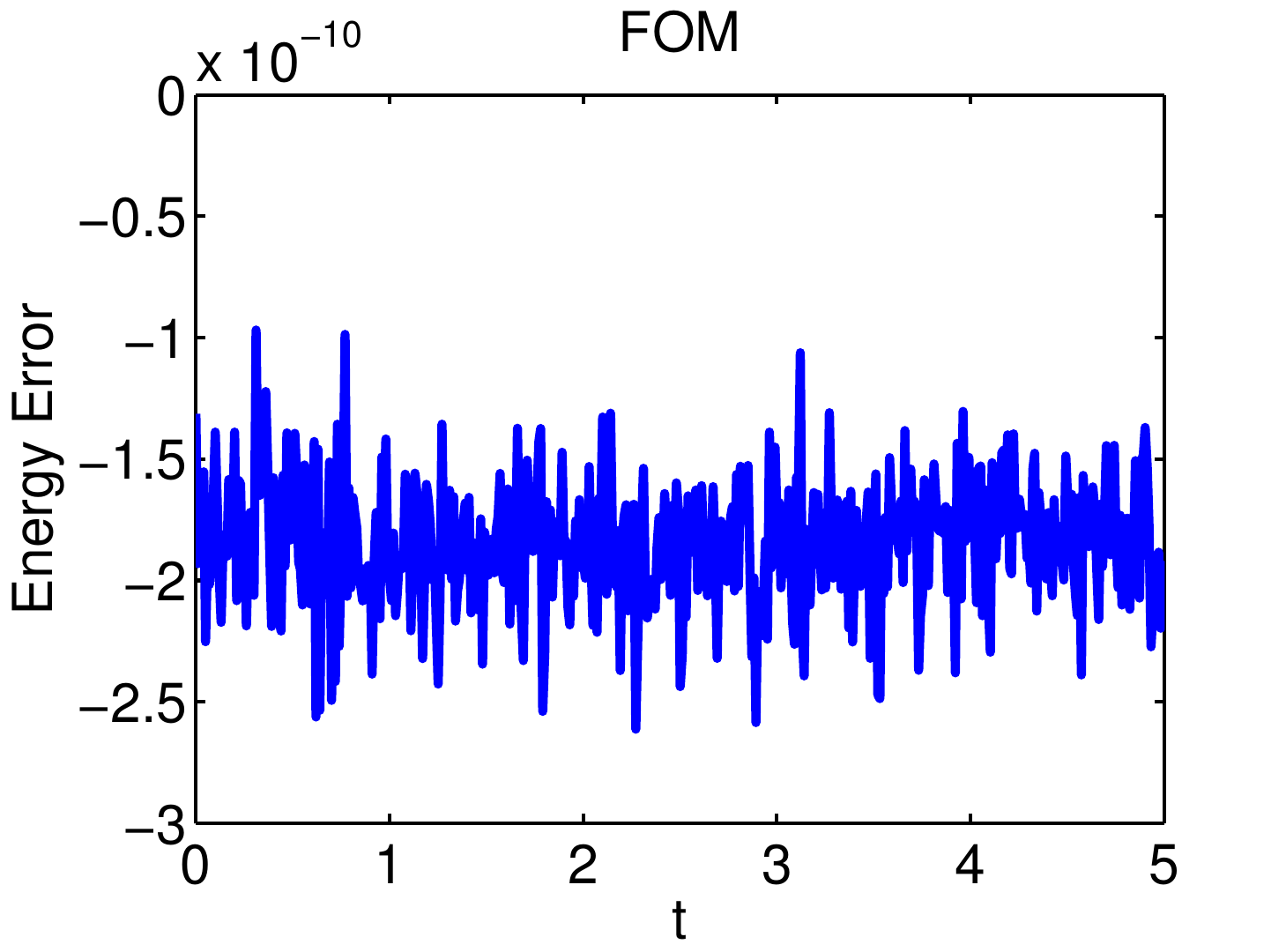}}
\subfloat{\includegraphics[scale=0.3]{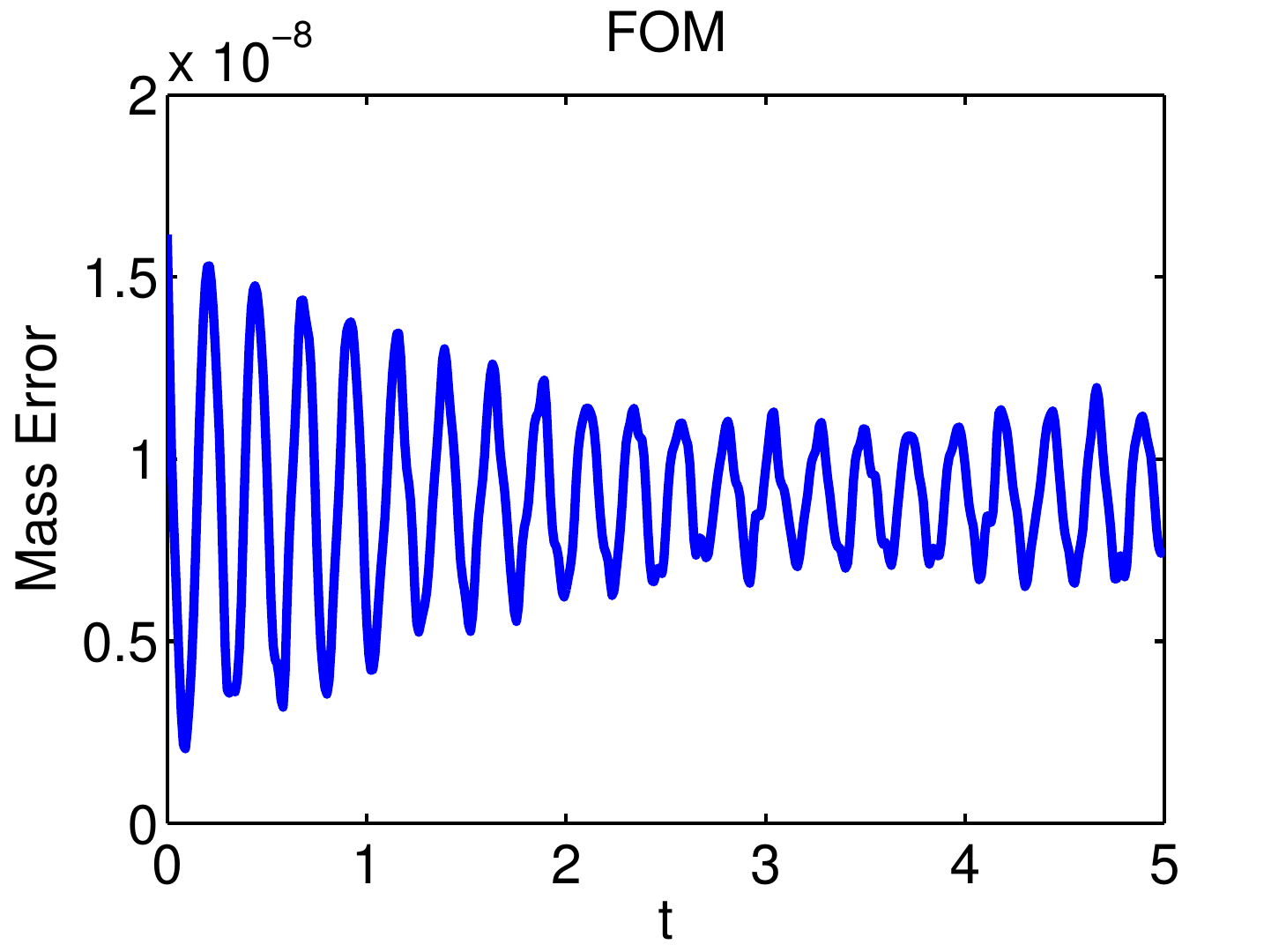}}

\subfloat{\includegraphics[scale=0.3]{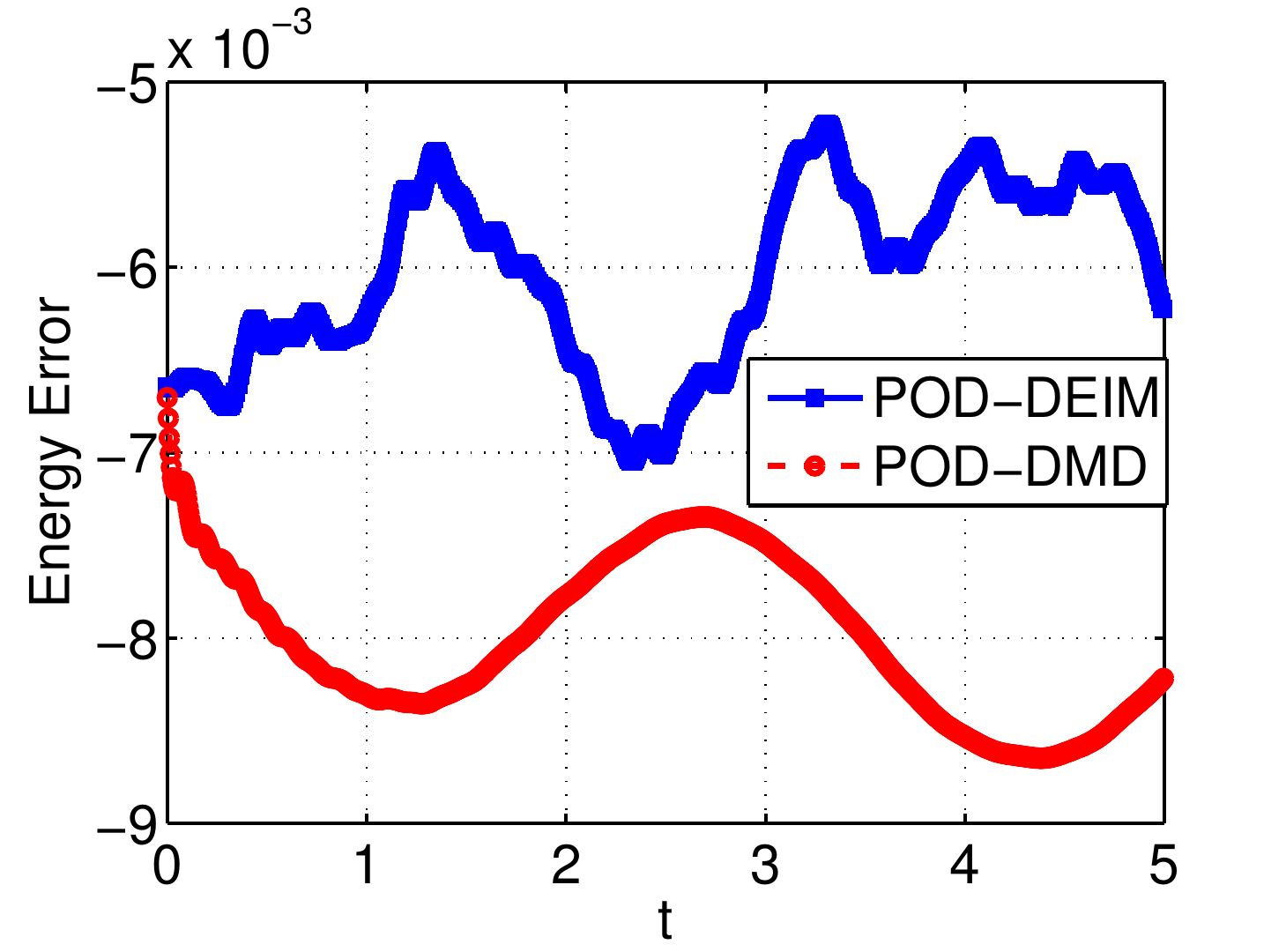}}
\subfloat{\includegraphics[scale=0.3]{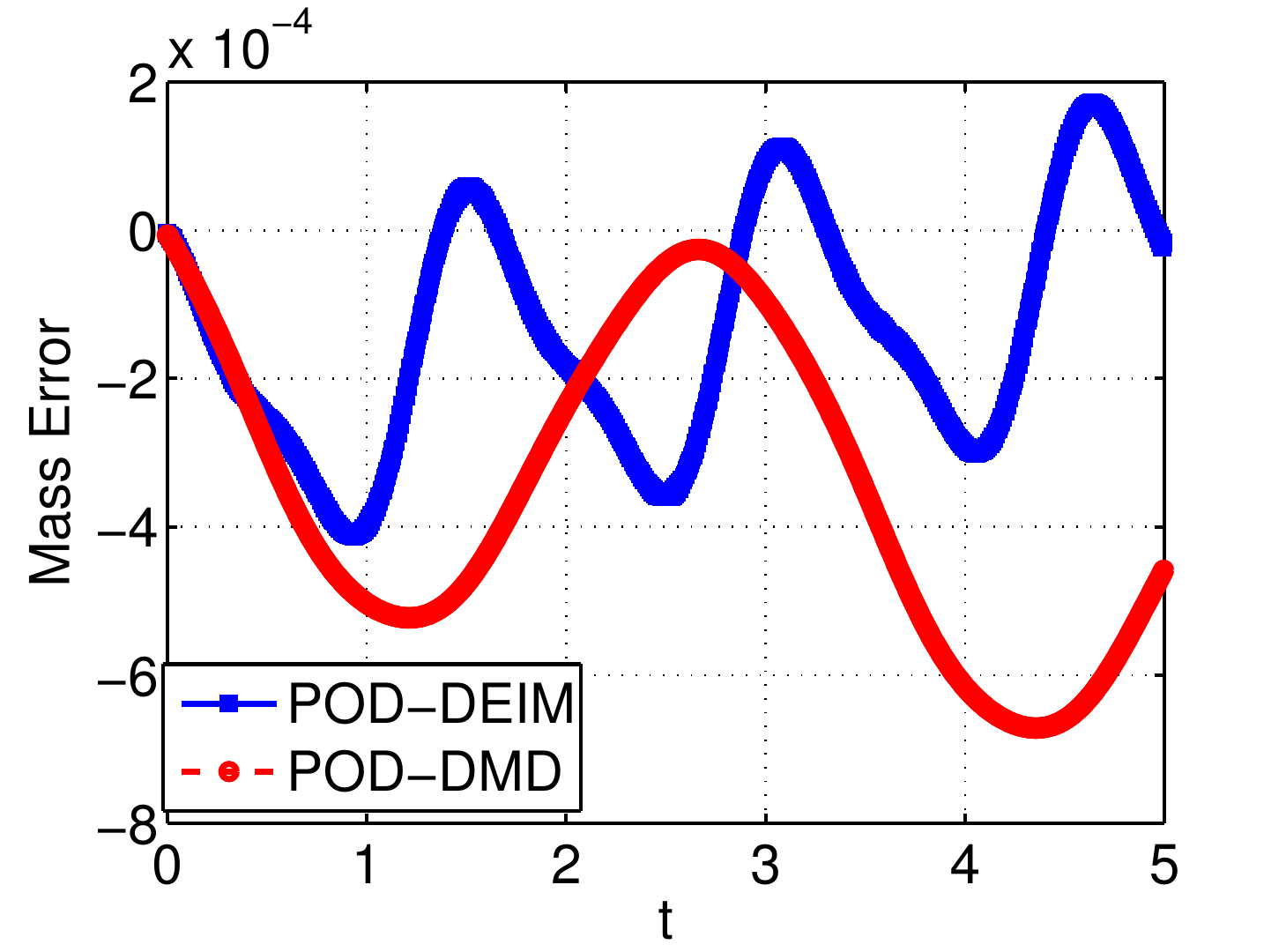}}
\caption{Example~\ref{ex1}: $L^{\infty}$-errors for the energy (left) and the mass (right) between initial ones by the solutions of FOM (top) and ROMs (bottom) with 10 POD and 15 DEIM/DMD modes\label{ex1_energymass}}
\end{figure}

On the other hand, for fixed number of POD modes $k=10$, the relative errors of the solution, of the discrete energy and of the discrete mass in Fig.~\ref{ex1_errors} decay monotonically with increasing number of the DEIM/DMD modes,  and after they stagnate. Similar results were obtained for 1D NLSE in \cite{Alla16}.
\begin{figure}[htb!]
\centering
\subfloat{\includegraphics[scale=0.3]{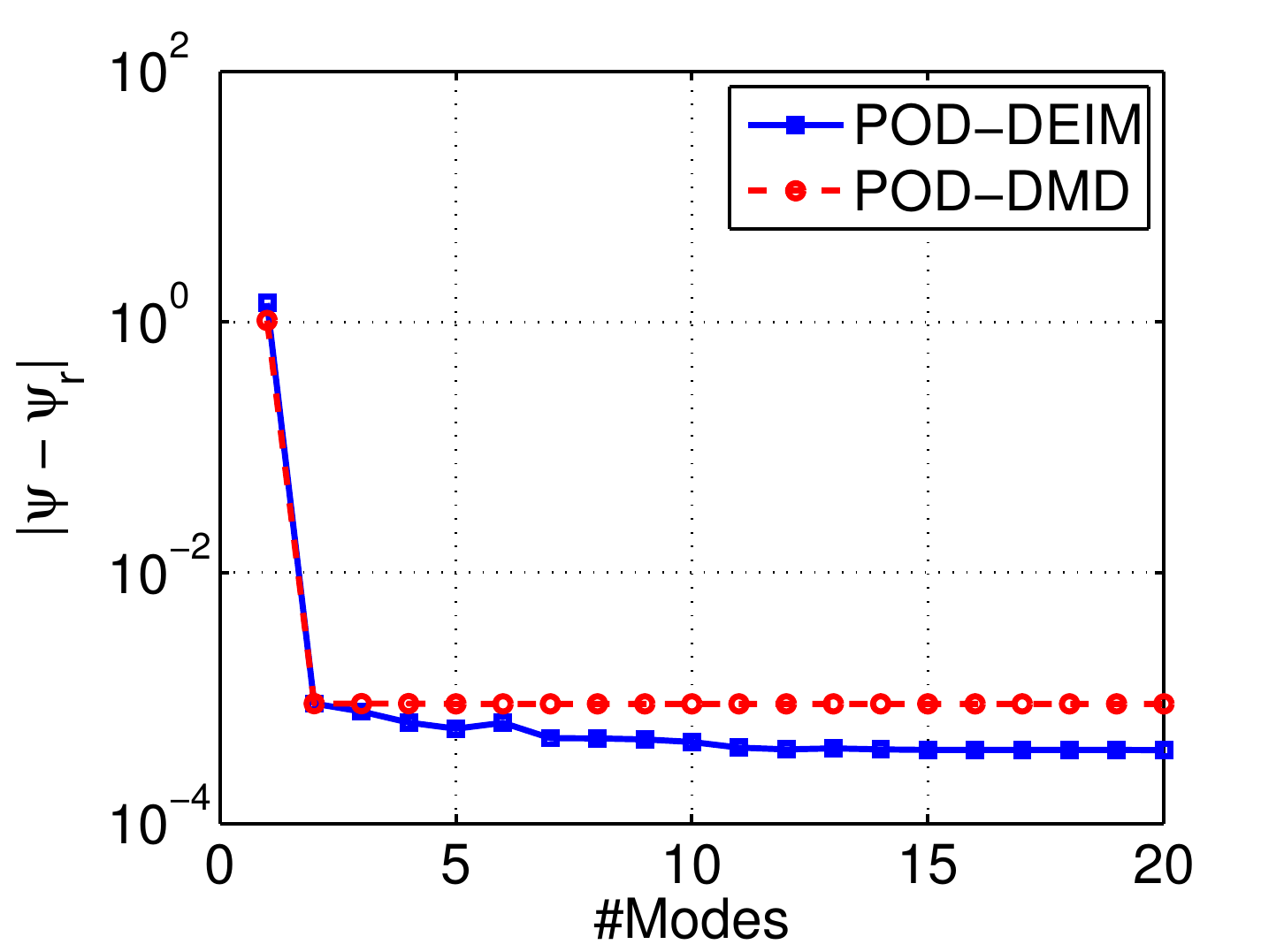}}
\subfloat{\includegraphics[scale=0.3]{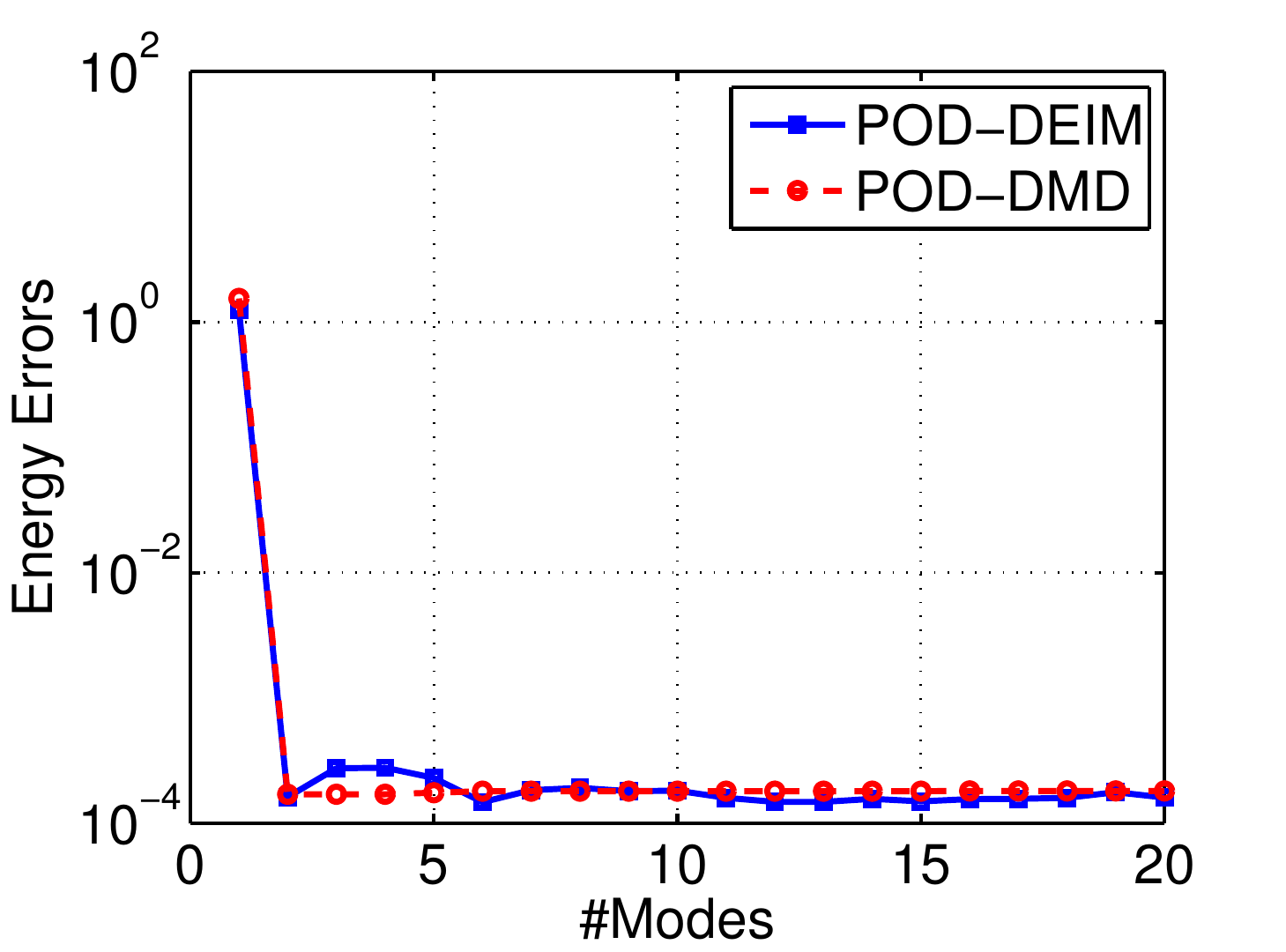}}
\subfloat{\includegraphics[scale=0.3]{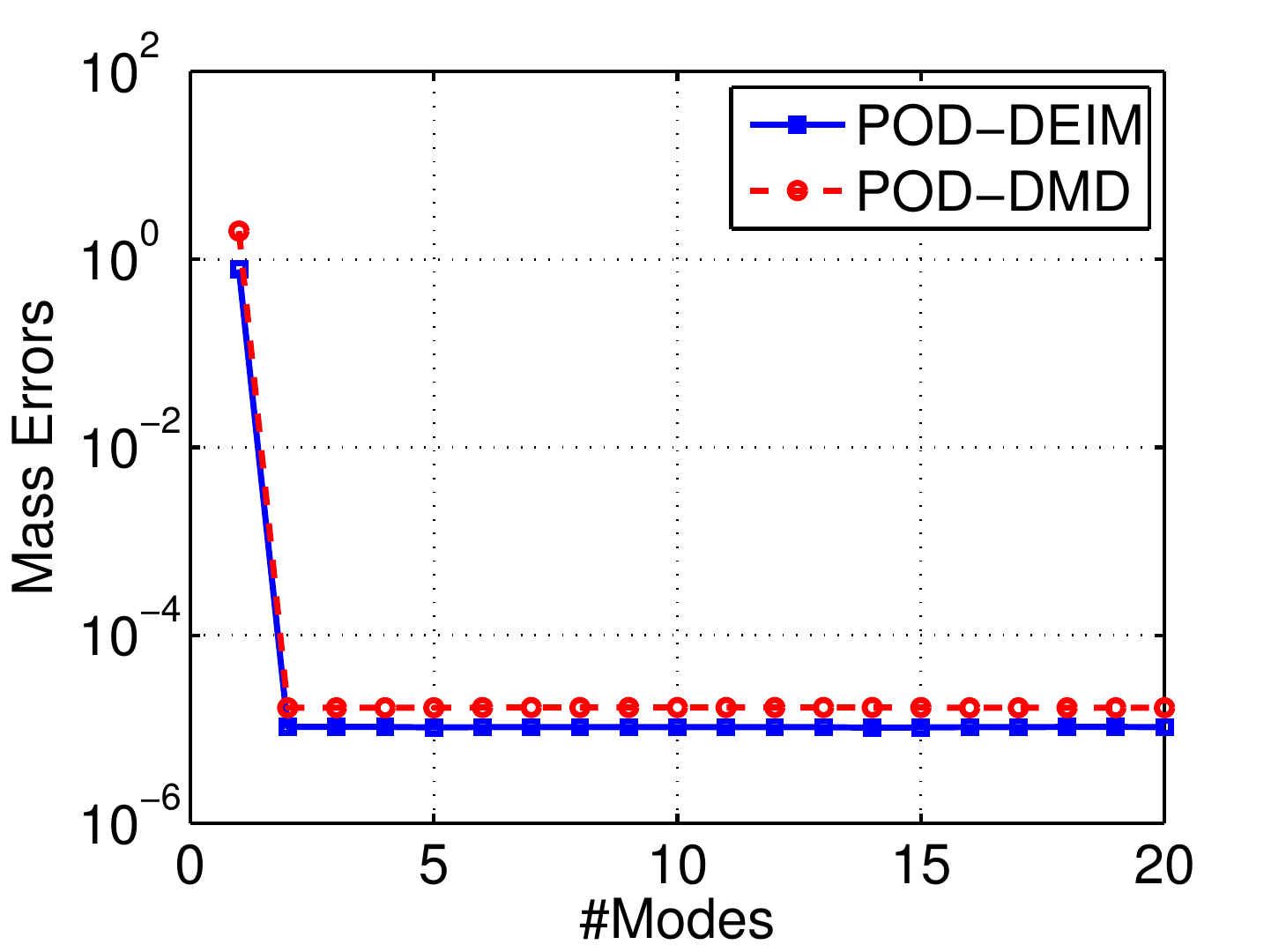}}
\caption{Example~\ref{ex1}: Relative $L^2$-$L^2$-errors for the solutions (left), and relative $L^\infty$-errors for the discrete energy (middle) and the discrete mass (right) by increasing number of DEIM/DMD modes, with 10 POD modes fixed\label{ex1_errors}}
\end{figure}

It can be clearly seen from Table~\ref{ex1_table} that the POD-DMD is much faster than the POD-DEIM, whereas there is no significant difference in the relative errors in Fig.~\ref{ex1_errors}. The advantage of the DMD lies in the fact that the ROM becomes linear, whereas in the case of DEIM, the ROM is still semi-linear and the nonlinearity have to be evaluated.

\begin{table}[H]
\centering
\caption{Example~\ref{ex1}: The computation time (in sec.) and speed-up factors for 10 POD and  15 DEIM/DMD modes}
\label{ex1_table}
\begin{tabular}{ l c c }
   & CPU Time (sec.) & Speedup \\
\hline
FOM            & 8024.7 & -  \\
ROM with DEIM  & 763.9 &  10.5   \\
ROM with DMD   & 8.2 &  974.3  \\
\hline
\end{tabular}
\end{table}

In Fig.~\ref{ex1_energymass_dt01}, we give the discrete energy and the discrete mass errors for the ROMs in $L^{\infty}$-norm, using a larger time-step size $\tau =0.01$ than for the FOM. The discrete energy and the discrete mass are still preserved  but with a lower accuracy.

\begin{figure}[htb!]
\centering
\includegraphics[scale=0.18]{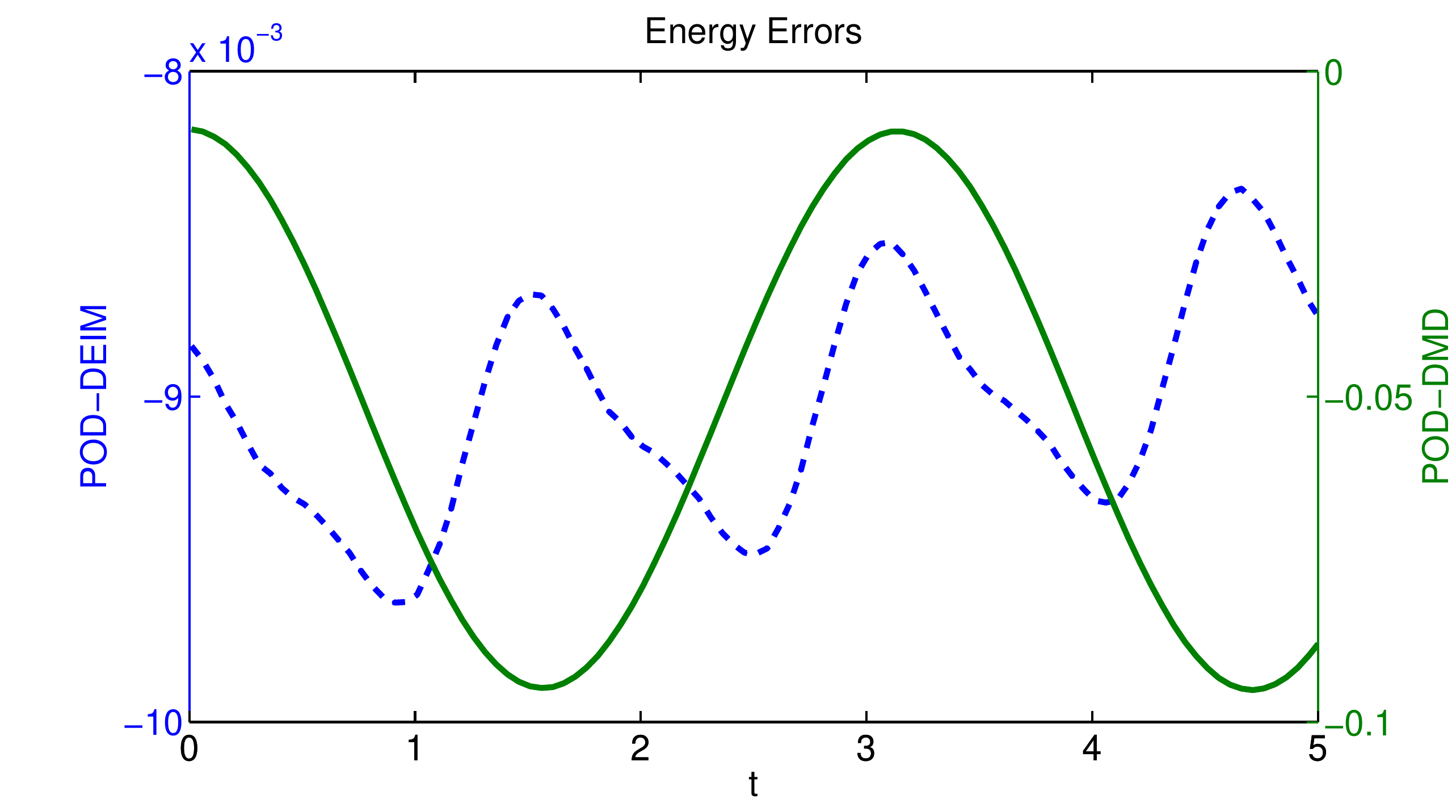}
\includegraphics[scale=0.18]{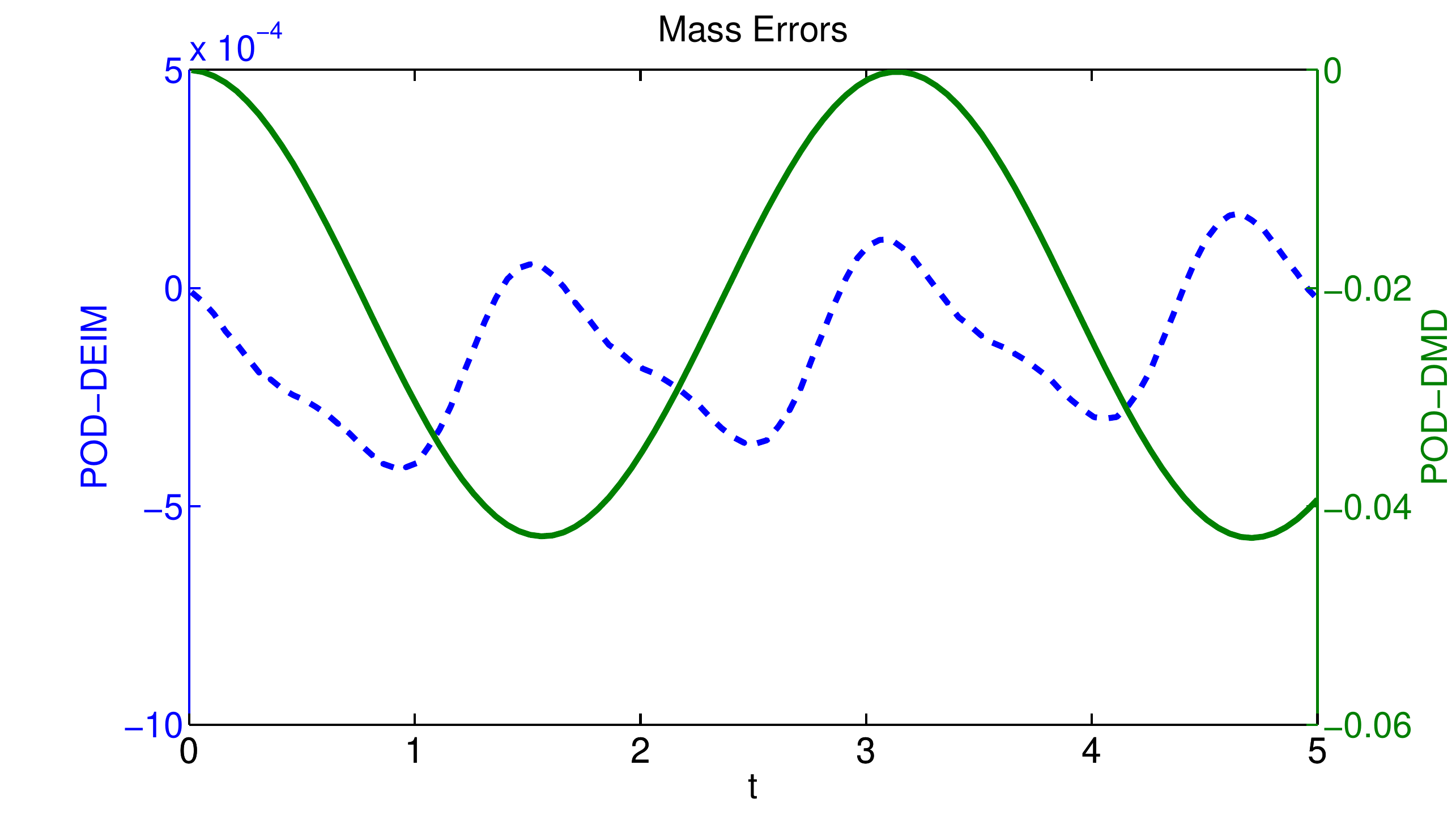}
\caption{Example~\ref{ex1}: $L^{\infty}$-errors for the discrete energy (left) and the discrete mass (right) by the ROMs with 10 POD, 15 DEIM/DMD modes, and  larger time-step size $\tau =0.01$\label{ex1_energymass_dt01}}
\end{figure}

We show in Figure \ref{ex1_energymassBE} that the discrete energy and the discrete mass are not conserved when a non-conservative time integrator like the backward Euler method is used.

\begin{figure}[htb!]
\centering
\subfloat{\includegraphics[scale=0.3]{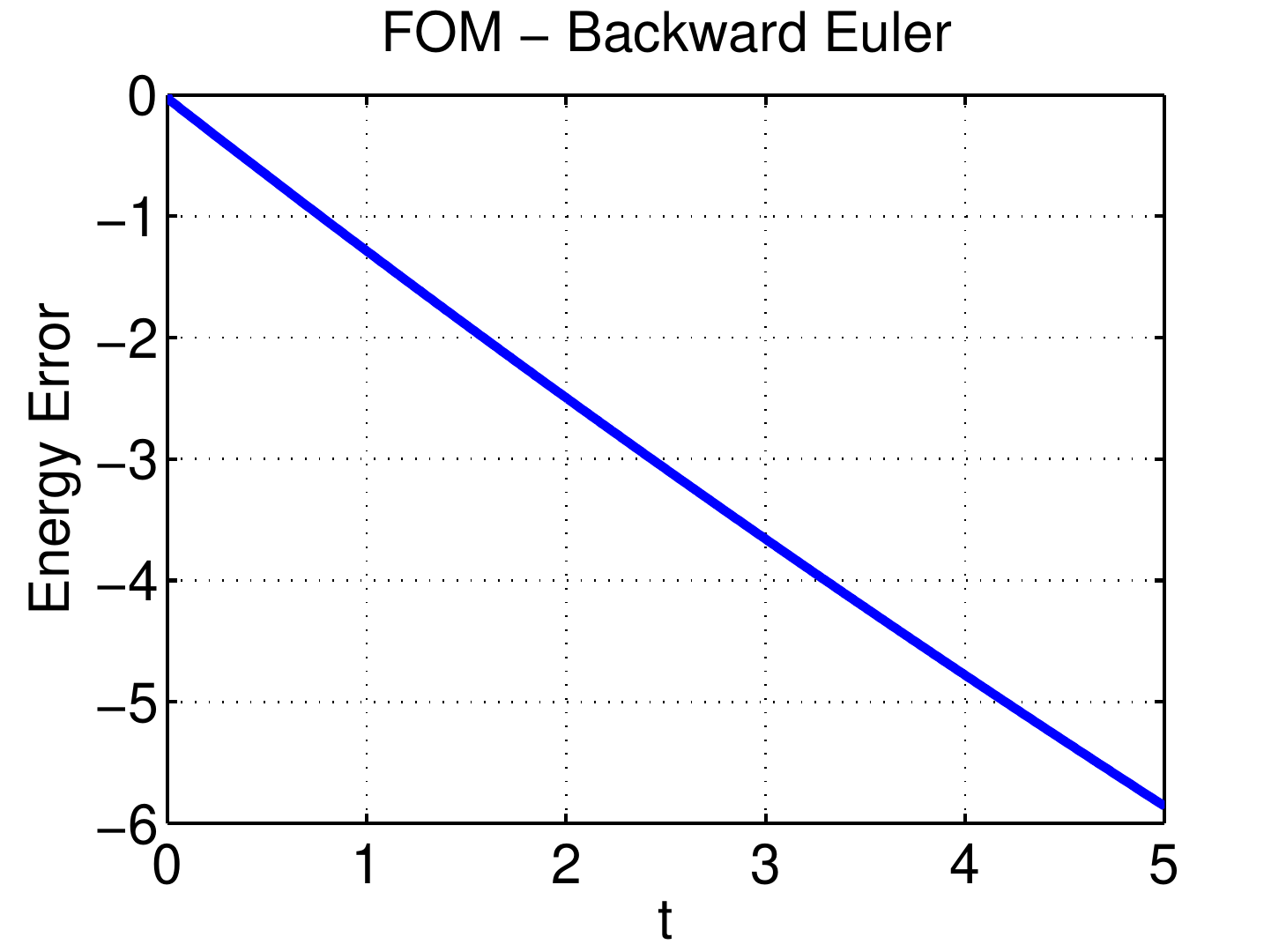}}
\subfloat{\includegraphics[scale=0.3]{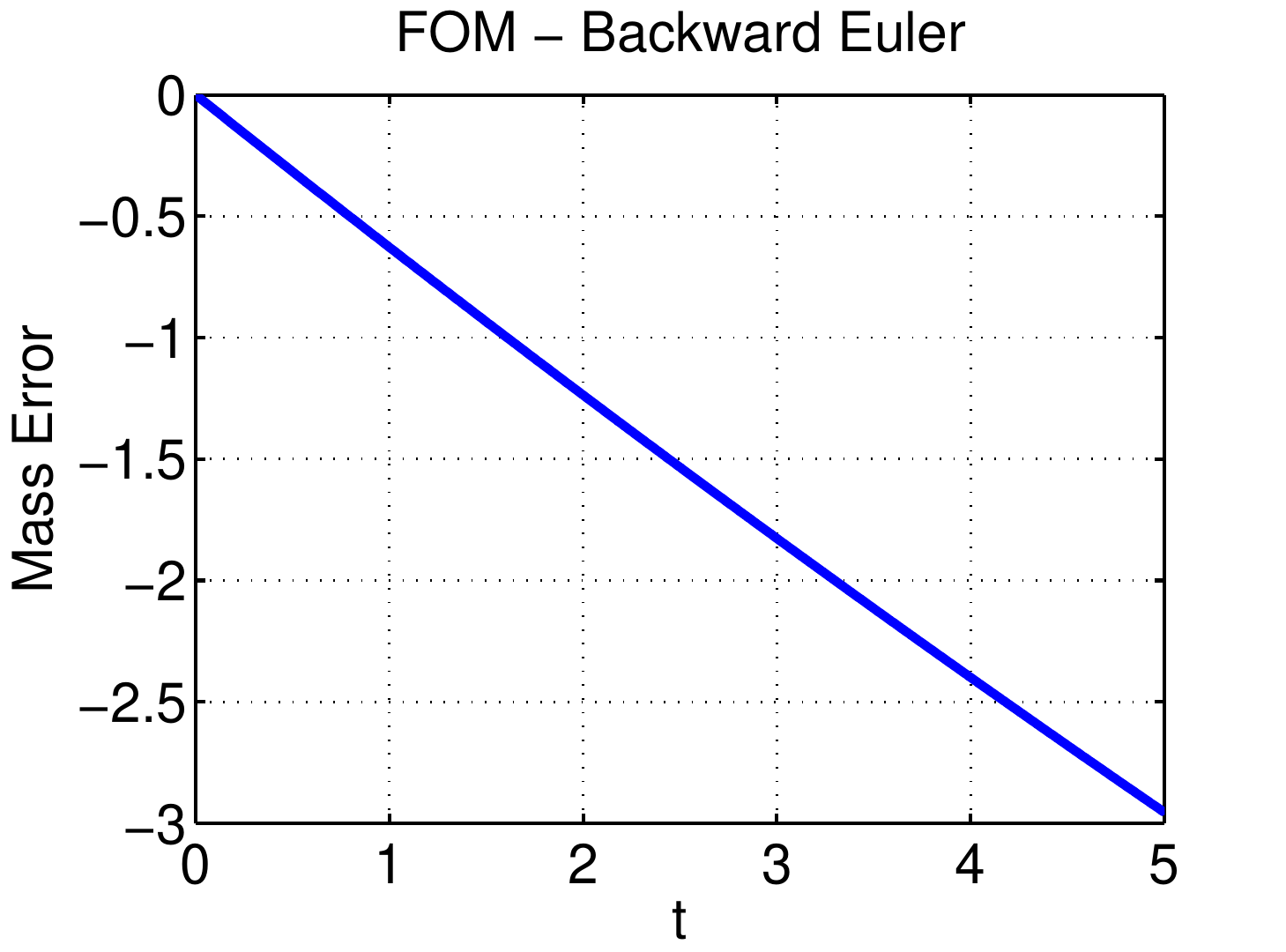}}
\caption{Example~\ref{ex1}: $L^{\infty}$-errors for the discrete energy (left) and the discrete mass (right) with the FOM solutions using backward Euler method.\label{ex1_energymassBE}}
\end{figure}

\subsection{NLSE with an external potential}
\label{ex2}

We consider for $\beta = 1$ the 2D defocusing NLSE \eqref{p1} with the harmonic  external trap potential  \cite{Galati13}

$$
V(x,y) = \frac{1}{2}(x^2  + 4y^2).
$$
The initial data is taken as the Gaussian

$$
\psi_0(\mathbf{x}) = \frac{1}{\sqrt{\pi}} e^{-\frac{(x^2 + y^2)}{2}},
$$
in the space-time domain $\Omega=[-8,8]^2$ and $T=3$, with the uniform spatial and temporal mesh sizes $h=0.5$ and $\tau =0.01$, respectively. The diffusion constant is set as $\alpha = 0.5$. The decay of the singular values in  Fig.~\ref{ex2_svd} is similar to the previous Example, i.e. NLSE without the external potential.
\begin{figure}[htb!]
\centering
\subfloat{\includegraphics[scale=0.35]{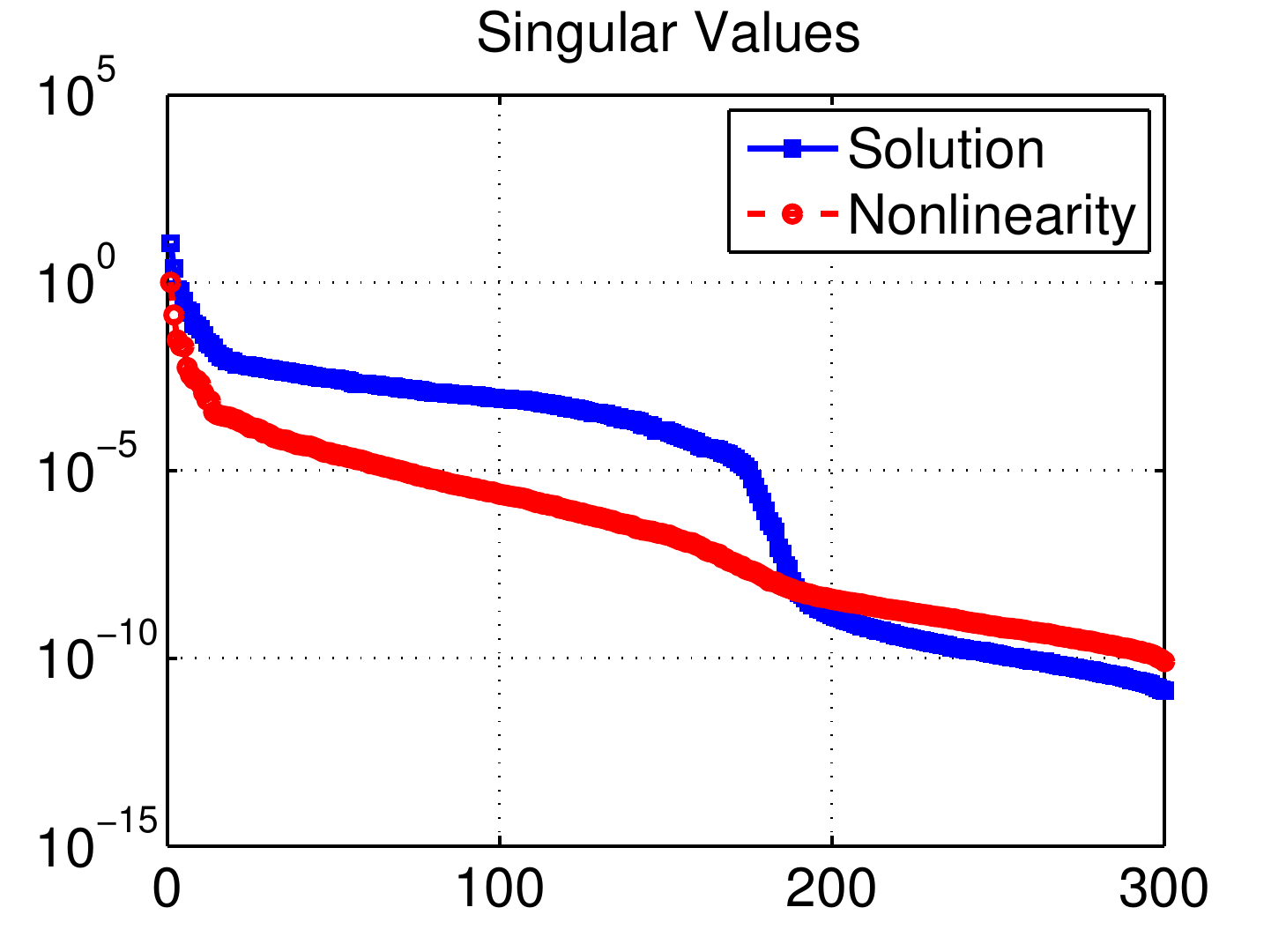}}
\caption{Example~\ref{ex2}: Decay of singular values\label{ex2_svd}}
\end{figure}

The POD-DEIM and POD-DMD solutions in Fig.~\ref{ex2_plots} have again the same accuracy, but requiring more POD, DEIM and DMD modes than the NLSE without the external potential. We set the number of POD modes $k=20$ according to the energy criterion $\varepsilon_{20}>0.9999$, further increase would not contribute a valuable accuracy.

\begin{figure}[htb!]
\centering
\subfloat{\includegraphics[scale=0.3]{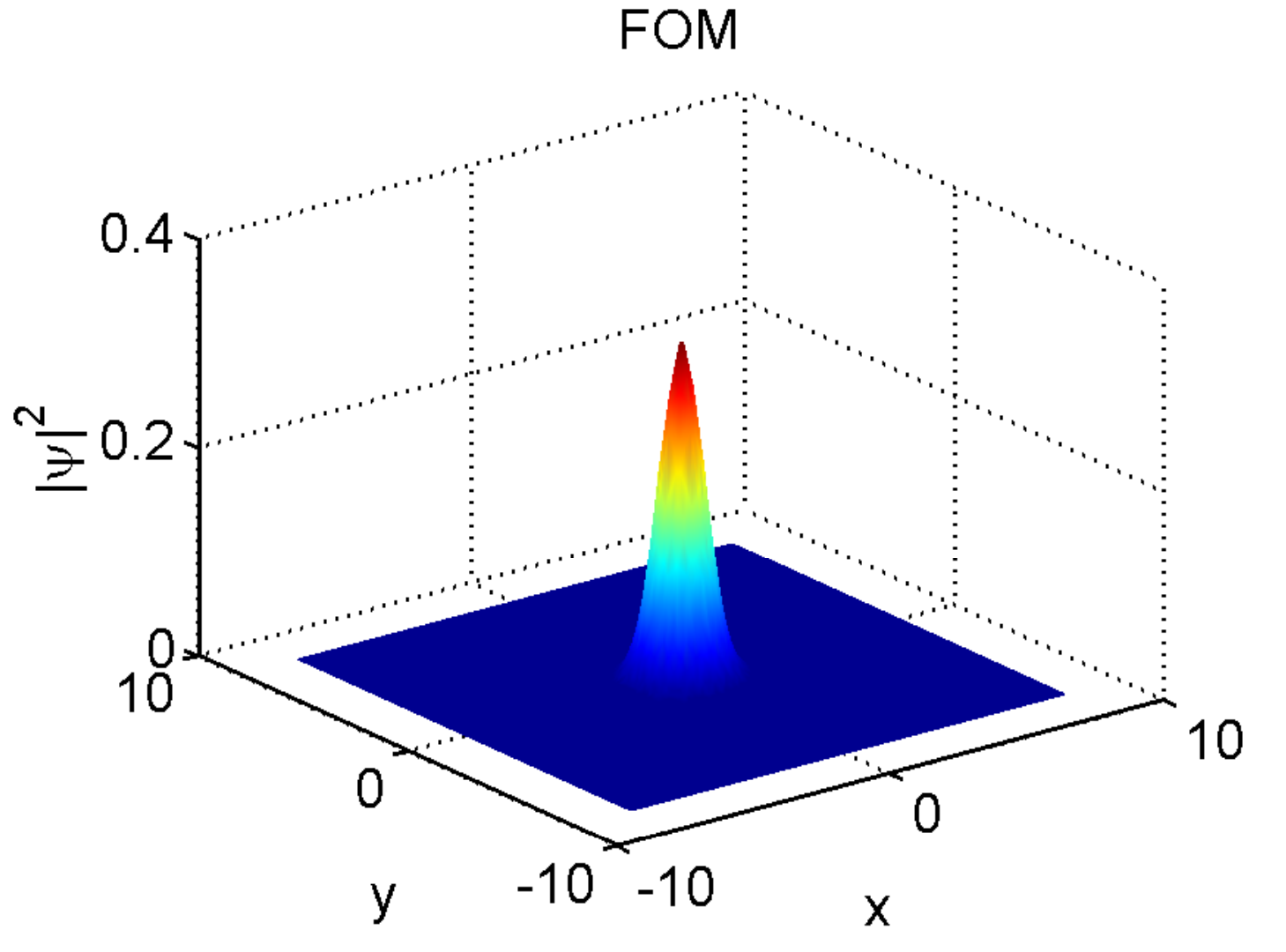}}
\subfloat{\includegraphics[scale=0.3]{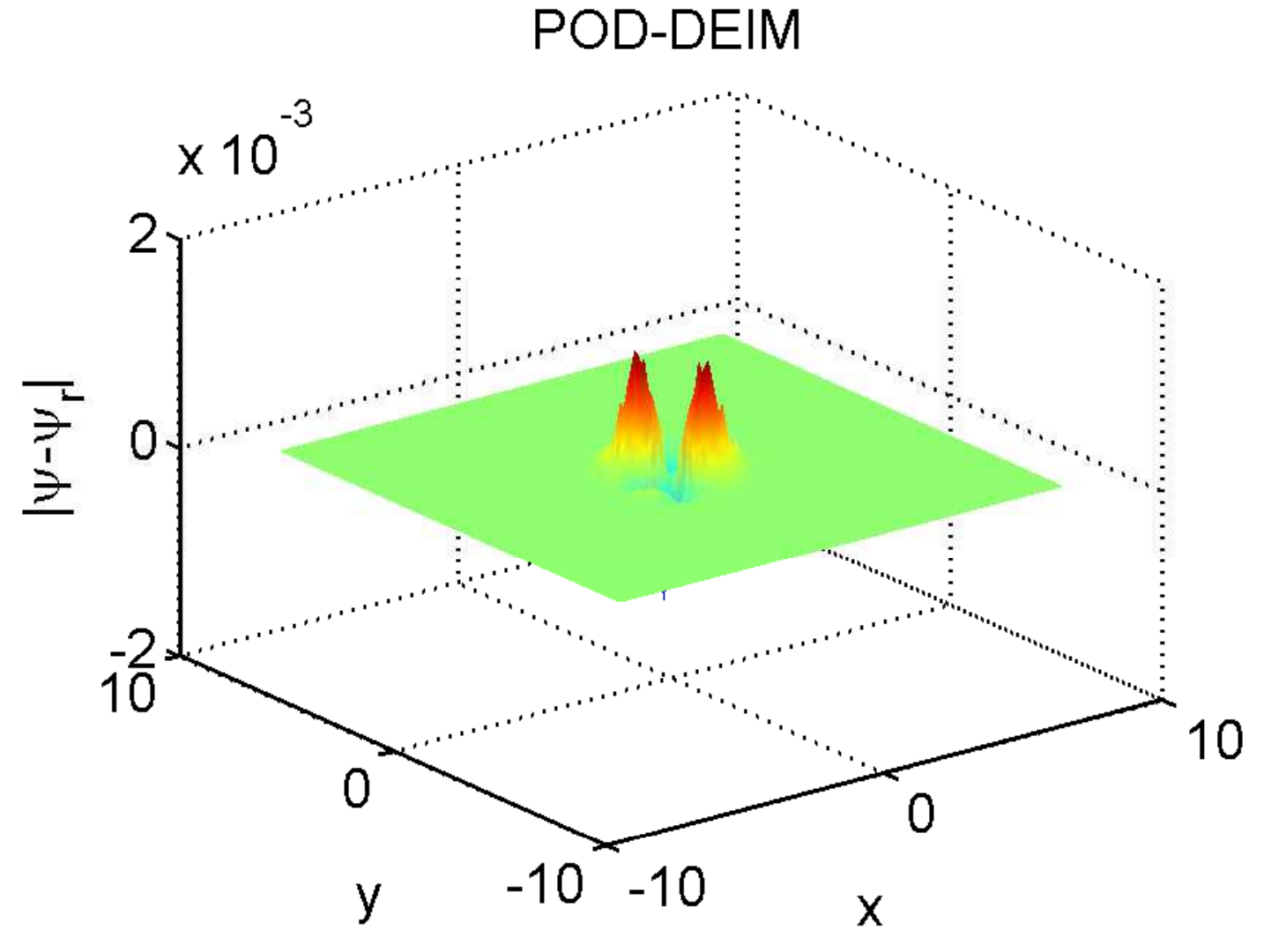}}
\subfloat{\includegraphics[scale=0.3]{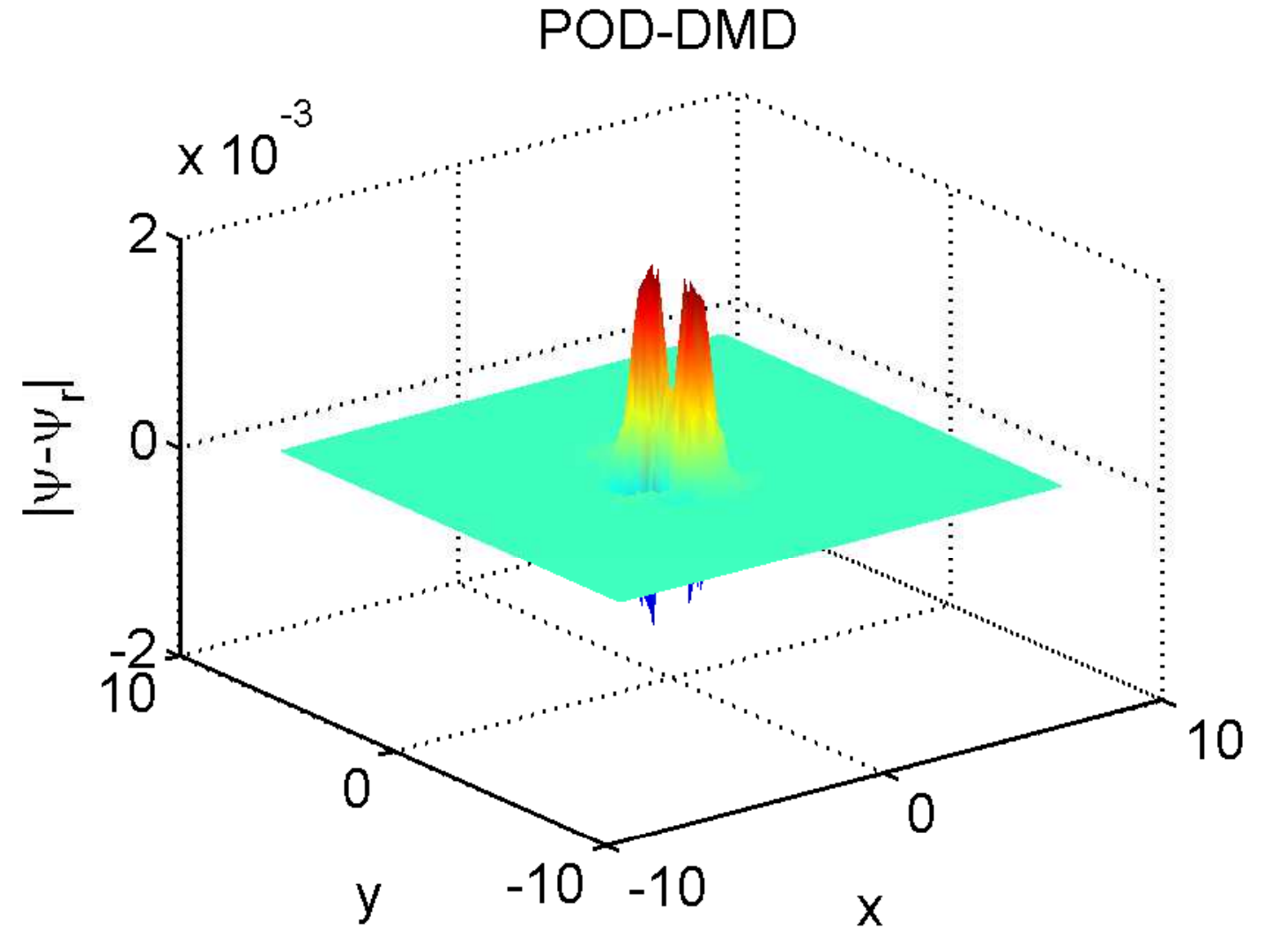}}
\caption{Example~\ref{ex2}: Solution profile by FOM (left), the errors between FOM and ROMs with 20 POD 20 DEIM (middle) and 20 POD and 20 DMD (right) modes, at the final time $T=3$\label{ex2_plots}}
\end{figure}

Again the discrete energy and the discrete mass are well preserved in Fig.~\ref{ex2_energymass} by the FOM and the ROMs.
\begin{figure}[htb!]
\centering
\subfloat{\includegraphics[scale=0.3]{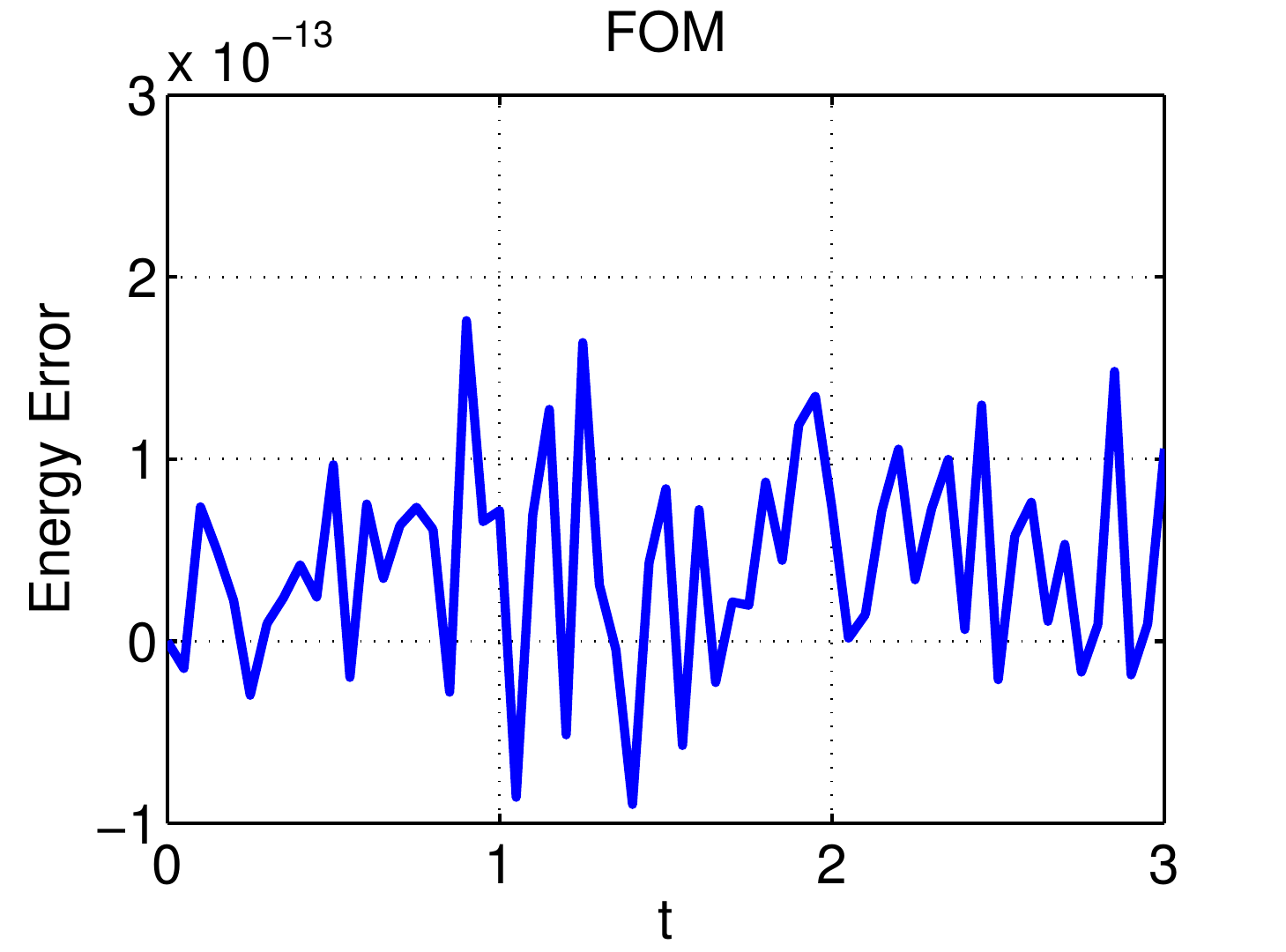}}
\subfloat{\includegraphics[scale=0.3]{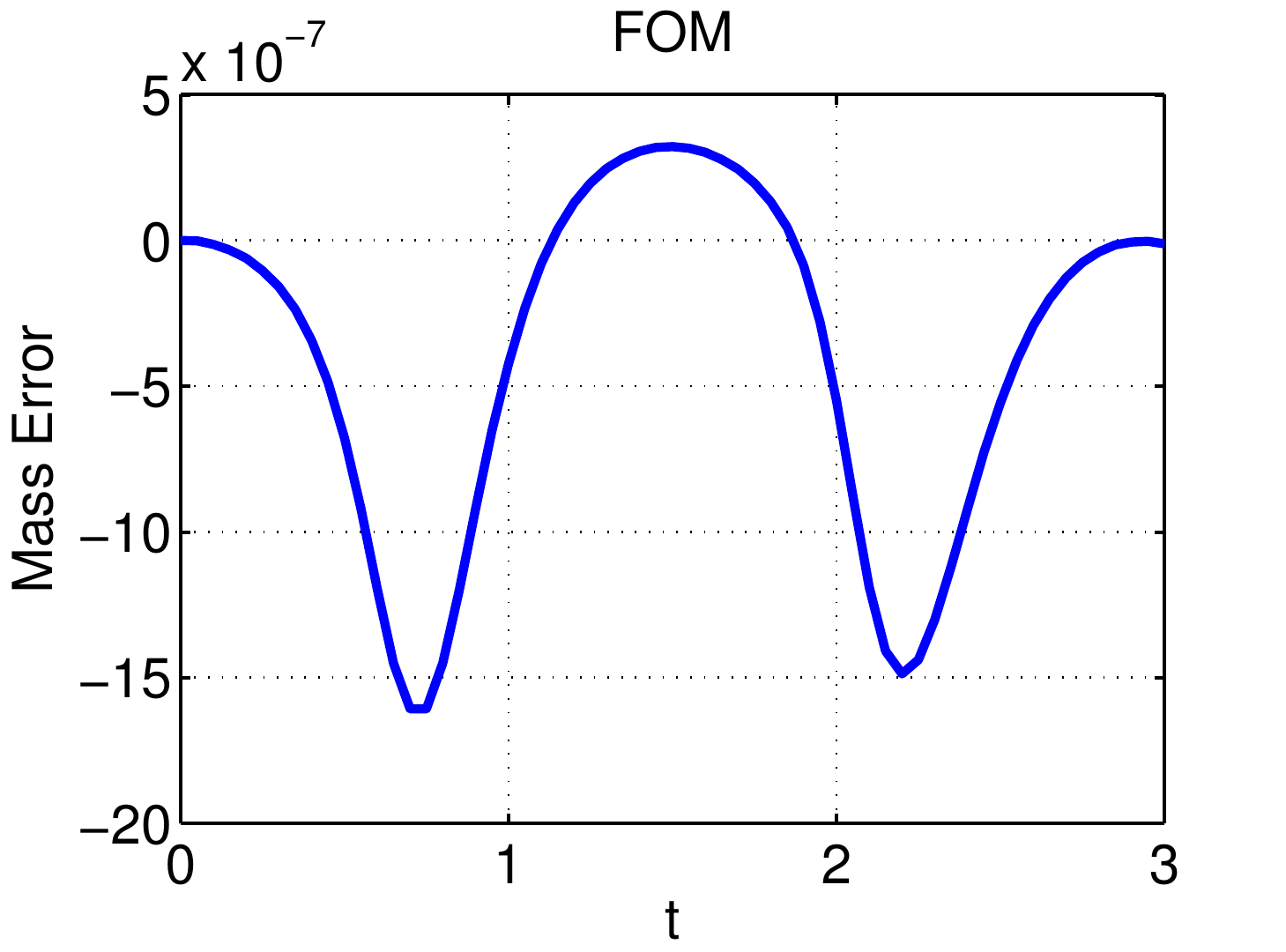}}

\subfloat{\includegraphics[scale=0.3]{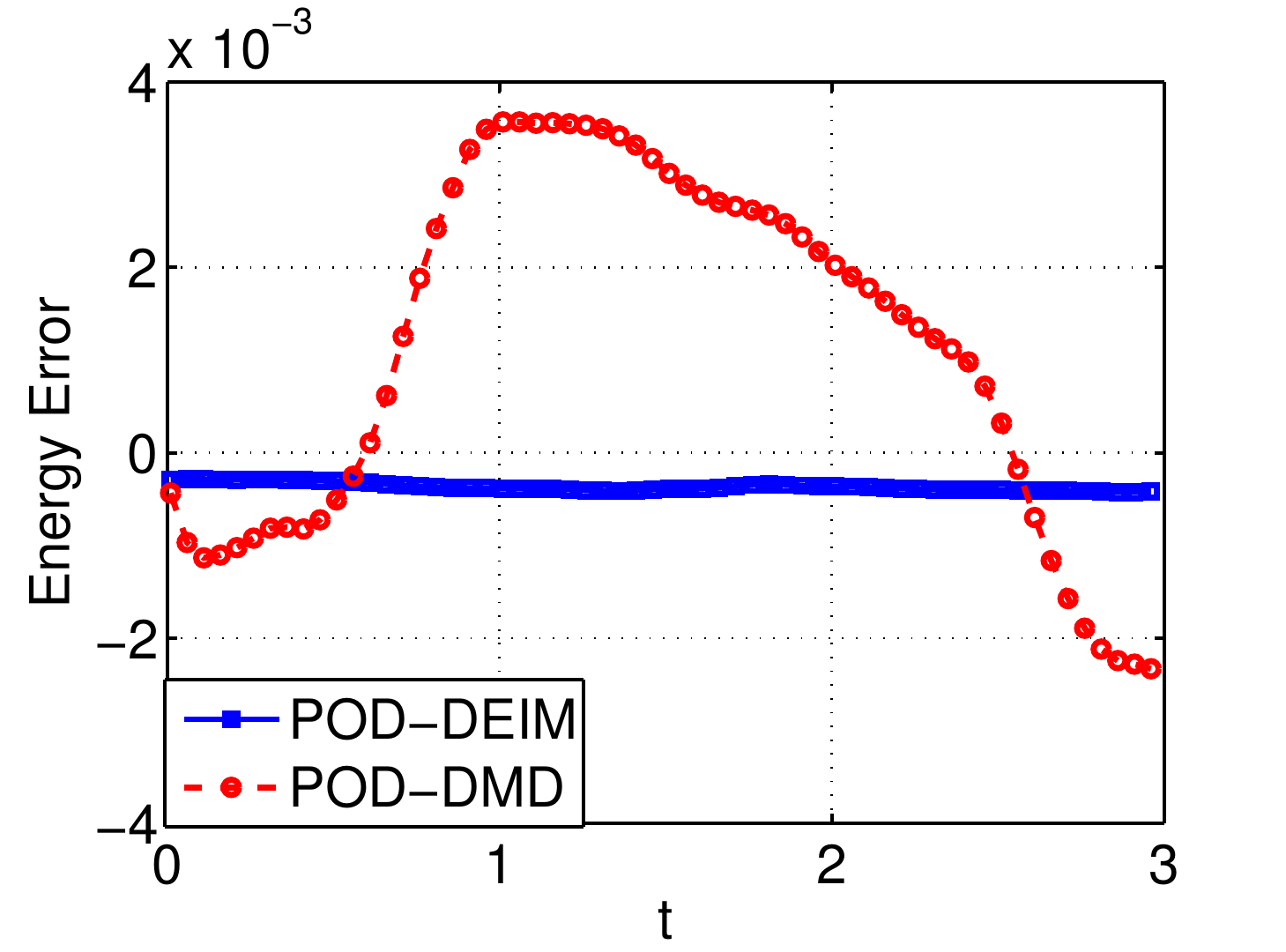}}
\subfloat{\includegraphics[scale=0.3]{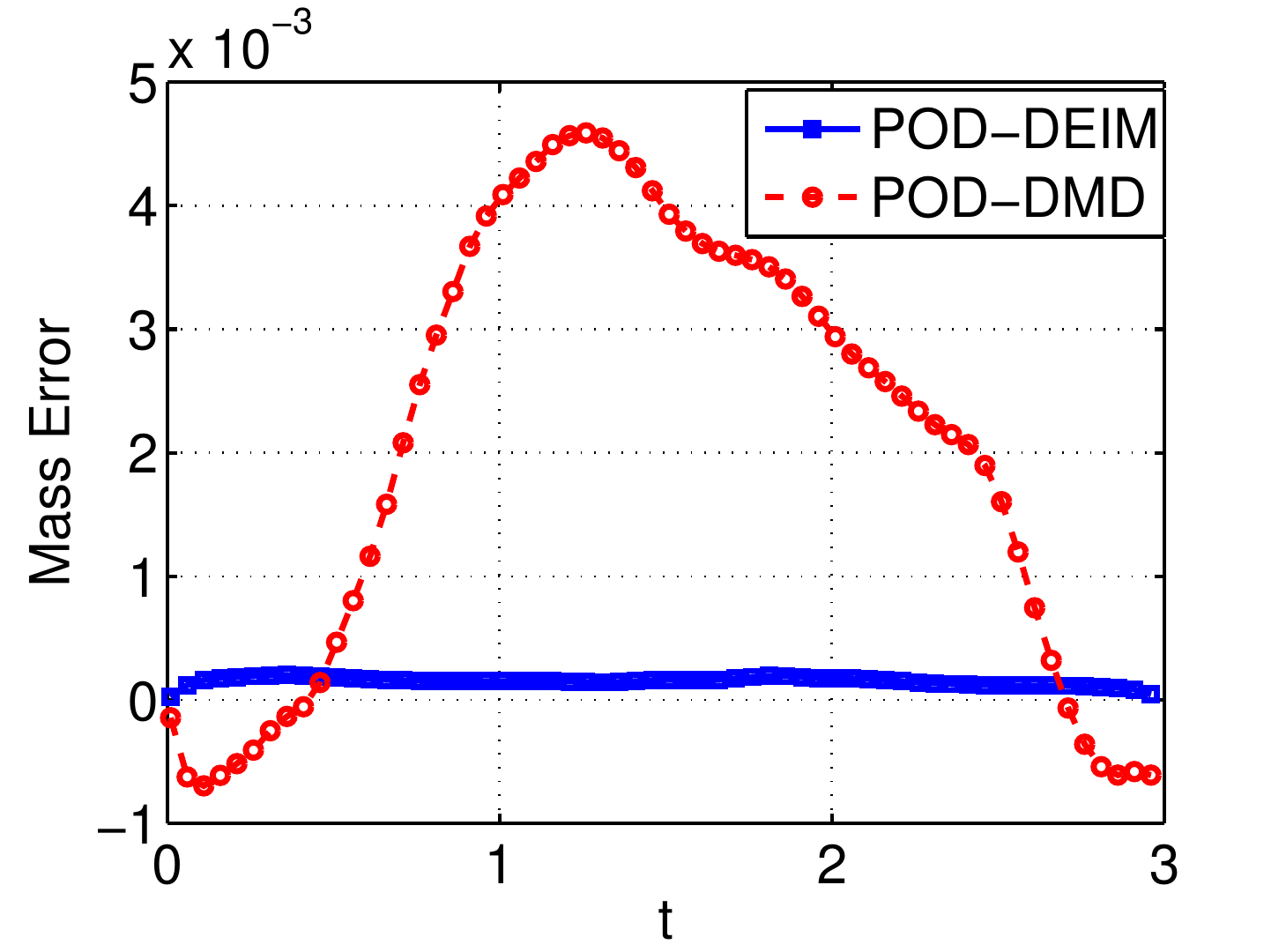}}
\caption{Example~\ref{ex2}: $L^{\infty}$-errors for the energy (left) and the mass (right) between initial ones by the solutions of FOM (top) and ROM (bottom) with 20 POD and 20 DEIM/DMD modes\label{ex2_energymass}}
\end{figure}

The solution, discrete energy and discrete mass errors in Fig.~\ref{ex2_errors} reach a plateau by increasing number of DEIM/DMD modes, and the POD-DEIM is much more accurate than the POD-DMD. The computational efficiency of the ROM with the DMD is clearly seen in Table~\ref{ex2_table}. Similar results are obtained for the one dimensional Burger's equation and NLSE in \cite{Alla16}.
\begin{figure}[htb!]
\centering
\subfloat{\includegraphics[scale=0.3]{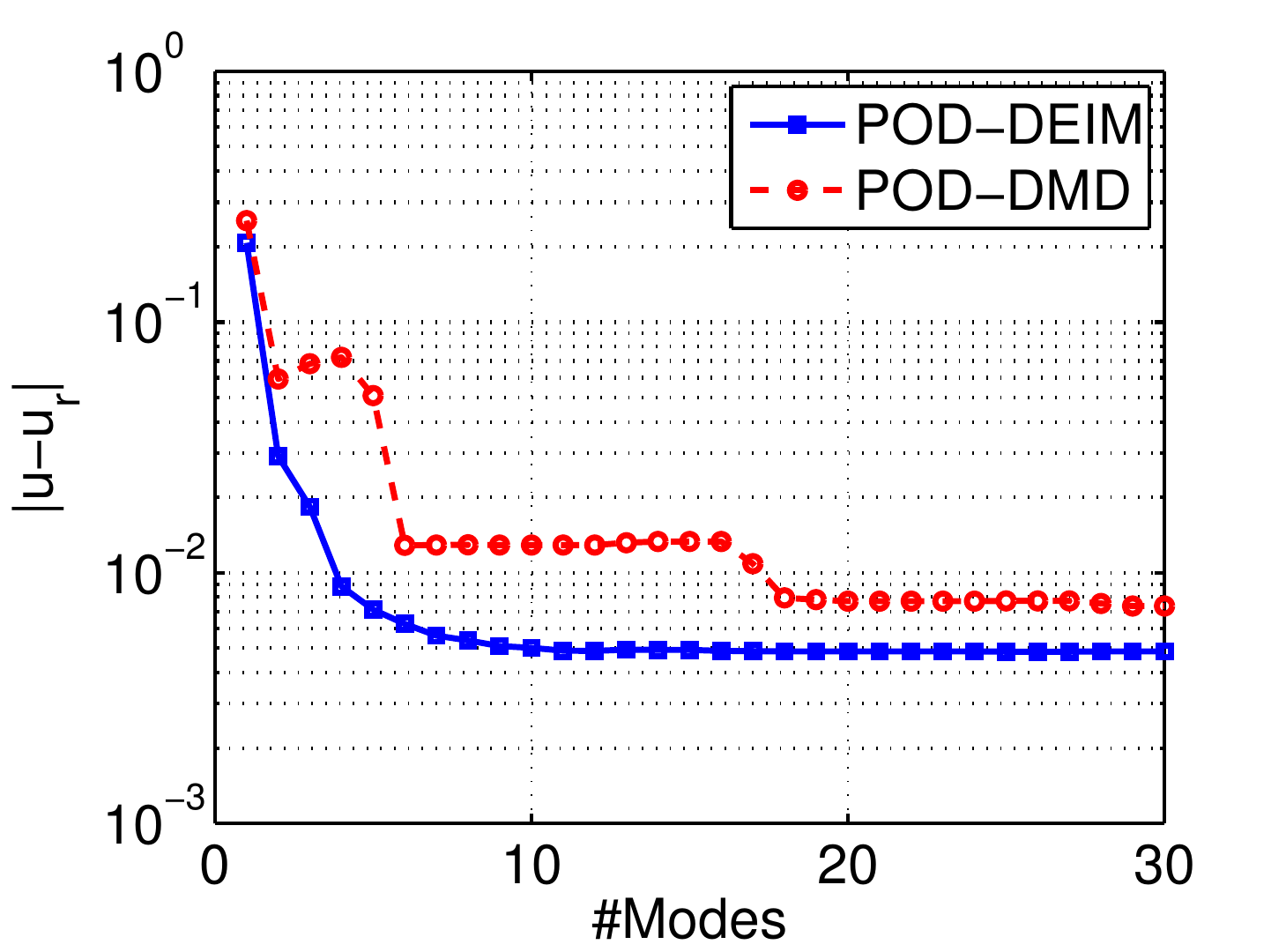}}
\subfloat{\includegraphics[scale=0.3]{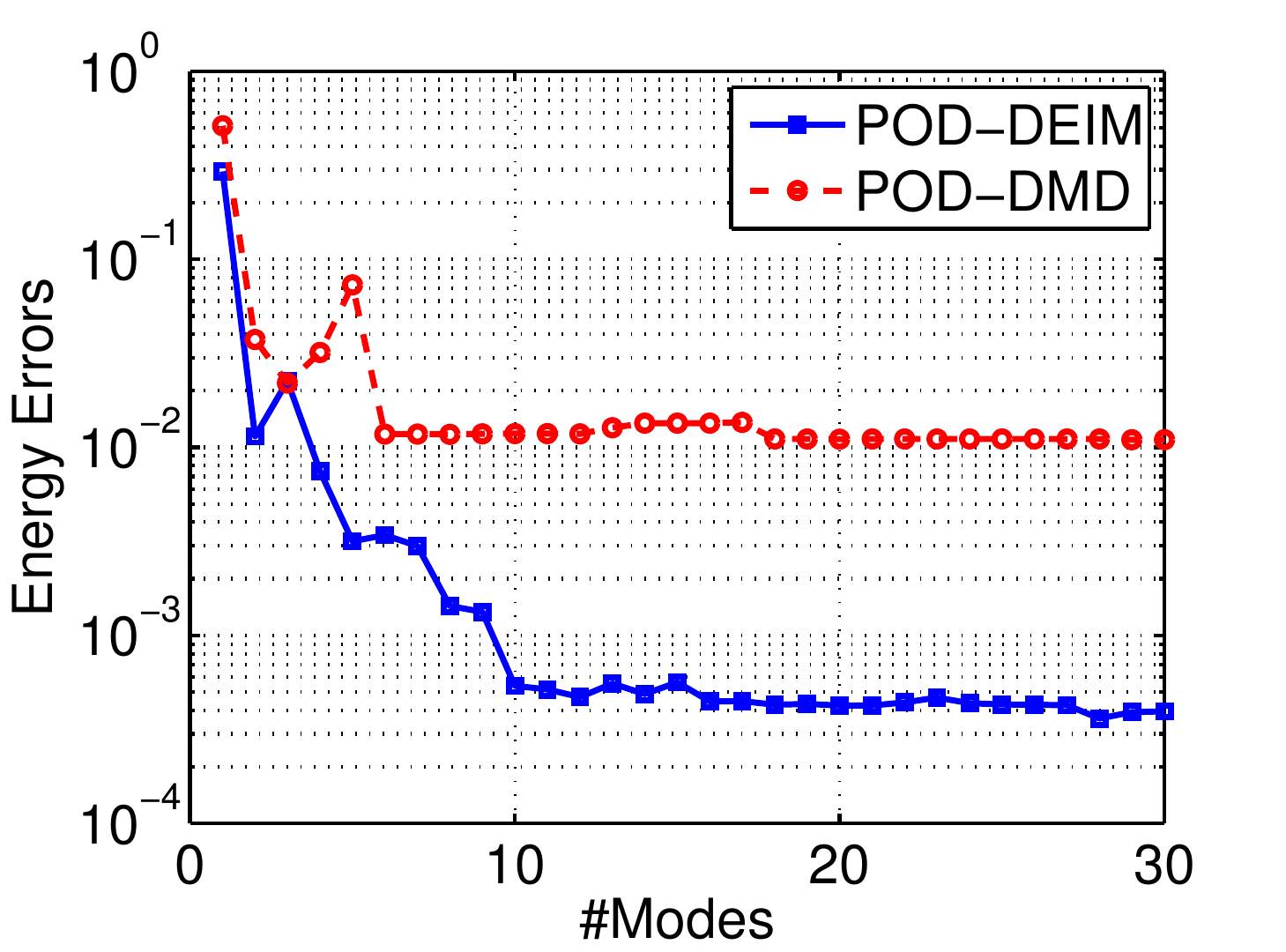}}
\subfloat{\includegraphics[scale=0.3]{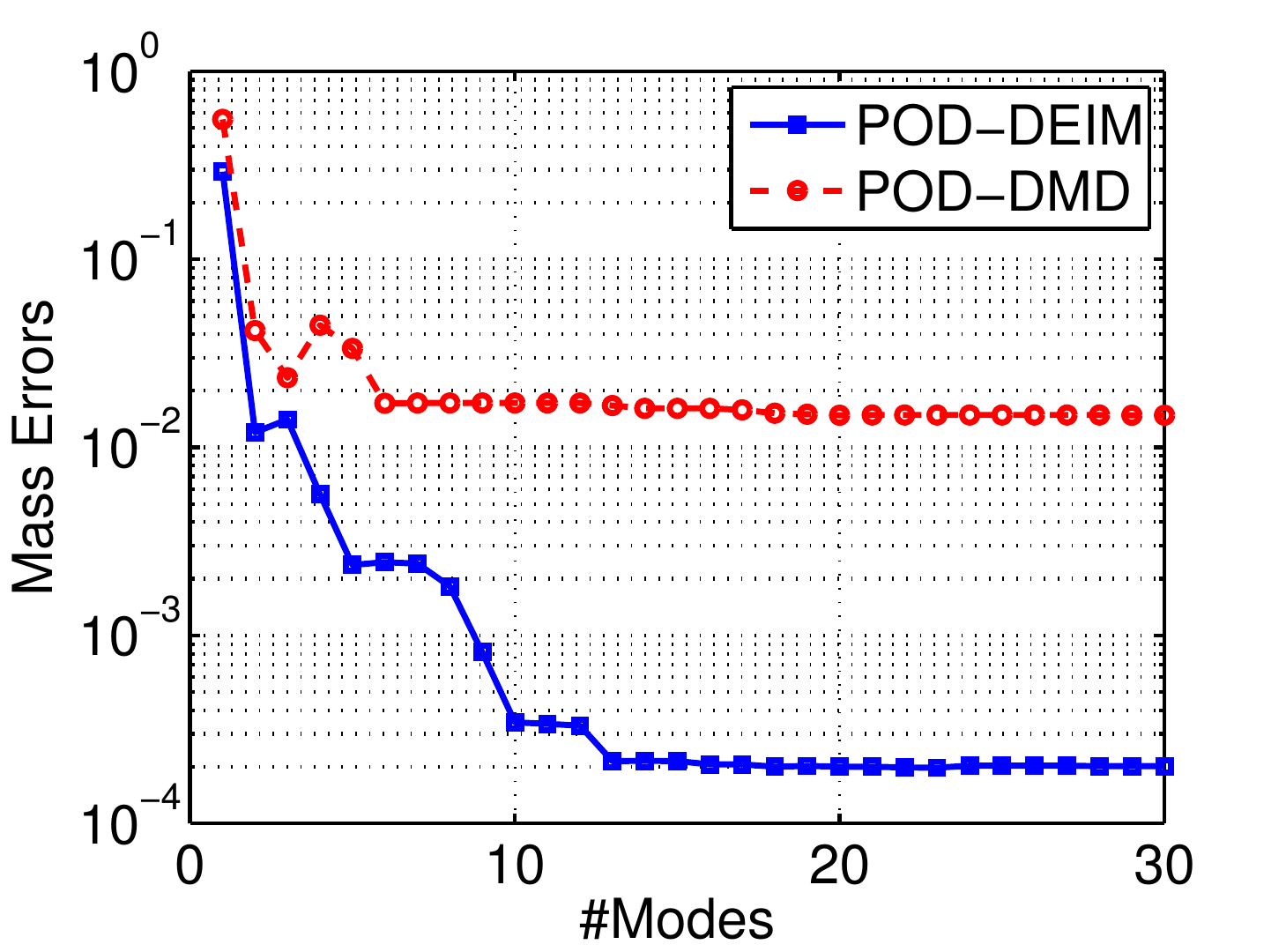}}
\caption{Example~\ref{ex2}: Relative $L^2$-$L^2$-errors for the solutions (left), and relative $L^\infty$-errors for the discrete energy (middle) and the discrete mass (right) by increasing number of DEIM/DMD modes, with 20 POD modes fixed\label{ex2_errors}}
\end{figure}

\begin{table}[H]
\centering
\caption{Example~\ref{ex2}: The computation time (in sec.) and speed-up factors for 20 POD/DEIM/DMD modes}
\label{ex2_table}
\begin{tabular}{ l c c }
   & CPU Time (sec.) & Speedup \\
\hline
FOM            & 309.7 & -  \\
ROM with DEIM  & 27.5 &  11.3   \\
ROM with DMD   & 0.3&  956.7  \\
\hline
\end{tabular}
\end{table}

\section{Conclusions}

We have performed a comparative study using the energy preserving ROM applied to 2D NLSE, showing the numerical efficiency and accuracy of the reduced order approximations and the complexity reduction of the nonlinear terms with DEIM and DMD. The discrete energy and  the discrete mass preservation of POD-DEIM and POD-DMD reduced order models have been shown to be accurate and efficient for capturing the spatio-temporal dynamics of the 2D NLSE  with substantial reduction in both dimension and computational time. This is clearly demonstrated in numerical simulations by the comparative computation times (speed-ups) and relative errors of the reduced order systems with respect to FOMs. The POD-DMD is always faster and the POD-DEIM is in general more accurate.\\

\noindent {\bf Acknowledgments:}
The authors would like to thank the reviewers for the comments and suggestions that
helped to  improve the manuscript.



\begin{thebibliography}{10}
\providecommand{\url}[1]{{#1}}
\providecommand{\urlprefix}{URL }
\expandafter\ifx\csname urlstyle\endcsname\relax
  \providecommand{\doi}[1]{DOI~\discretionary{}{}{}#1}\else
  \providecommand{\doi}{DOI~\discretionary{}{}{}\begingroup
  \urlstyle{rm}\Url}\fi

\bibitem{Hesthaven16}
Afkham, B.M., Hesthaven, J.S.: Structure preserving model reduction of
  parametric {Hamiltonian} systems.
\newblock SIAM Journal on Scientific Computing \textbf{39}(6), A2616--A2644
  (2017).
\newblock \doi{10.1137/17M1111991}

\bibitem{Alla16a}
Alla, A., Kutz, J.: Randomized model order reduction.
\newblock ArXiv e-prints  (2016)

\bibitem{Alla16}
Alla, A., Kutz, J.N.: Nonlinear model order reduction via dynamic mode
  decomposition.
\newblock SIAM Journal on Scientific Computing \textbf{39}(5), B778--B796
  (2017).
\newblock \doi{10.1137/16M1059308}

\bibitem{Heinkenschloss14}
Antil, H., Heinkenschloss, M., Sorensen Danny, C.: Application of the discrete
  empirical interpolation method to reduced order modeling of nonlinear and
  parametric systems.
\newblock In: A.~Quarteroni, G.~Rozza (eds.) Reduced Order Methods for Modeling
  and Computational Reduction, \emph{MS \& A - Modeling, Simulation and
  Applications}, vol.~9, pp. 101--136. Springer International Publishing
  (2014).
\newblock \doi{10.1007/978-3-319-02090-7\_4}

\bibitem{Antoine13}
Antoine, X., Bao, W., Besse, C.: Computational methods for the dynamics of the
  nonlinear {Schr{\"o}dinger/Gross-Pitaevskii} equations.
\newblock Computer Physics Communications \textbf{184}(12), 2621 -- 2633
  (2013).
\newblock \doi{10.1016/j.cpc.2013.07.012}

\bibitem{Antoinegpelab1}
Antoine, X., Duboscq, R.: {GPELab}, a {Matlab} toolbox to solve
  {Gross-Pitaevskii} equations i: Computation of stationary solutions.
\newblock Computer Physics Communications \textbf{185}(11), 2969 -- 2991
  (2014).
\newblock \doi{10.1016/j.cpc.2014.06.026}

\bibitem{Antoinegpelab2}
Antoine, X., Duboscq, R.: {GPELab}, a matlab toolbox to solve
  {Gross-Pitaevskii} equations ii: Dynamics and stochastic simulations.
\newblock Computer Physics Communications \textbf{193}, 95 -- 117 (2015).
\newblock \doi{10.1016/j.cpc.2015.03.012}

\bibitem{Arnold02}
Arnold, D.N., Brezzi, F., Cockburn, B., Marini, L.D.: Unified analysis of
  discontinuous {Galerkin} methods for elliptic problems.
\newblock SIAM Journal on Numerical Analysis \textbf{39}(5), 1749--1779 (2002).
\newblock \doi{10.1137/S0036142901384162}

\bibitem{Astrid08}
Astrid, P., Weiland, S., Willcox, K., Backx, T.: Missing point estimation in
  models described by proper orthogonal decomposition.
\newblock IEEE Transactions on Automatic Control \textbf{53}(10), 2237--2251
  (2008).
\newblock \doi{10.1109/TAC.2008.2006102}

\bibitem{Bao13a}
Bao, W., Cai, Y.: Mathematical theory and numerical methods for
  {B}ose--{E}instein condensation.
\newblock Kinetic and Related Models \textbf{6}(1), 1--135 (2013).
\newblock \doi{10.3934/krm.2013.6.1}

\bibitem{Barrault04}
Barrault, M., Maday, Y., Nguyen, N.C., Patera, A.T.: An empirical interpolation
  method: application to efficient reduced-basis discretization of partial
  differential equations.
\newblock Comptes Rendus Mathematique \textbf{339}(9), 667--672 (2004).
\newblock \doi{10.1016/j.crma.2004.08.006}

\bibitem{Beattie11}
Beattie, C., Gugercin, S.: Structure-preserving model reduction for nonlinear
  port-{H}amiltonian systems.
\newblock In: 2011 50th IEEE Conference on Decision and Control and European
  Control Conference, pp. 6564--6569 (2011).
\newblock \doi{10.1109/CDC.2011.6161504}

\bibitem{Bistrian17}
Bistrian, D.A., Navon, I.M.: Randomized dynamic mode decomposition for
  nonintrusive reduced order modelling.
\newblock International Journal for Numerical Methods in Engineering  (2017).
\newblock \doi{10.1002/nme.5499}

\bibitem{Bridges06}
Bridges, T.J., Reich, S.: Numerical methods for {H}amiltonian {PDEs}.
\newblock Journal of Physics A: Mathematical and General \textbf{39}(19),
  5287--5320 (2006).
\newblock \doi{10.1088/0305-4470/39/19/S02}

\bibitem{Carlberg13}
Carlberg, K., Farhat, C., Cortial, J., Amsallem, D.: The \{GNAT\} method for
  nonlinear model reduction: Effective implementation and application to
  computational fluid dynamics and turbulent flows.
\newblock Journal of Computational Physics \textbf{242}, 623 -- 647 (2013).
\newblock \doi{10.1016/j.jcp.2013.02.028}

\bibitem{Carlberg15}
Carlberg, K., Tuminaro, R., Boggs, P.: Preserving {L}agrangian structure in
  nonlinear model reduction with application to structural dynamics.
\newblock SIAM J. Sci. Comput. \textbf{37}(2), B153--B184 (2015).
\newblock \doi{10.1137/140959602}

\bibitem{Celledoni14}
Celledoni, E., Owren, B., Sun, Y.: The minimal stage, energy preserving
  {R}unge-{K}utta method for polynomial {H}amiltonian systems is the averaged
  vector field method.
\newblock Math. Comp. \textbf{83}(288), 1689--1700 (2014).
\newblock \doi{10.1090/S0025-5718-2014-02805-6}

\bibitem{Celledoni12ped}
{Celledoni, E. and Grimm, V. and McLachlan, R. I. and McLaren, D. I. and
  O'Neale, D. J. and Owren, B. and Quispel, G. R. W.}: {Preserving energy resp.
  dissipation in numerical PDEs using the {"Average Vector Field"} method.}
\newblock J. Comput. Physics \textbf{231}, 6770--6789 (2012).
\newblock \doi{10.1016/j.jcp.2012.06.022}

\bibitem{Charnyi17}
Charnyi, S., Heister, T., Olshanskii, M.A., Rebholz, L.G.: On conservation laws
  of {Navier-Stokes Galerkin} discretizations.
\newblock Journal of Computational Physics \textbf{337}, 289 -- 308 (2017).
\newblock \doi{10.1016/j.jcp.2017.02.039}

\bibitem{Chaturantabu16}
Chaturantabut, S., Beattie, C., Gugercin, S.: Structure-preserving model
  reduction for nonlinear {Port-Hamiltonian} systems.
\newblock SIAM Journal on Scientific Computing \textbf{38}(5), B837--B865
  (2016).
\newblock \doi{10.1137/15M1055085}

\bibitem{chaturantabut10nmr}
Chaturantabut, S., Sorensen, D.C.: Nonlinear model reduction via discrete
  empirical interpolation.
\newblock SIAM J. SCI. COMPUT. \textbf{32}(5), 2737--2764 (2010).
\newblock \doi{10.1137/090766498}

\bibitem{Chen02}
Chen, J.B., Qin, M.Z., Tang, Y.F.: Symplectic and multi-symplectic methods for
  the nonlinear {Schr\"odinger} equation.
\newblock Computers \& Mathematics with Applications \textbf{43}(8), 1095--1106
  (2002).
\newblock \doi{10.1016/S0898-1221(02)80015-3}

\bibitem{Cohen11}
Cohen, D., Hairer, E.: Linear energy-preserving integrators for {Poisson}
  systems.
\newblock BIT Numerical Mathematics \textbf{51}(1), 91--101 (2011).
\newblock \doi{10.1007/s10543-011-0310-z}

\bibitem{Faou09}
Debussche, A., Faou, E.: Modified energy for split-step methods applied to the
  linear schrödinger equation.
\newblock SIAM Journal on Numerical Analysis \textbf{47}(5), 3705--3719 (2009).
\newblock \doi{10.1137/080744578}

\bibitem{Drohmann12}
Drohmann, M., Haasdonk, B., Ohlberger, M.: Reduced basis approximation for
  nonlinear parametrized evolution equations based on empirical operator
  interpolation.
\newblock SIAM Journal on Scientific Computing \textbf{34}(2), A937--A969
  (2012).
\newblock \doi{10.1137/10081157X}

\bibitem{Erichson16}
Erichson, N.B., Donovan, C.: Randomized low-rank dynamic mode decomposition for
  motion detection.
\newblock Computer Vision and Image Understanding \textbf{146}, 40 -- 50
  (2016).
\newblock \doi{https://doi.org/10.1016/j.cviu.2016.02.005}

\bibitem{Everson95}
Everson, R., Sirovich, L.: Karhunen--{L}o\`{e}ve procedure for gappy data.
\newblock J. Opt. Soc. Am. A \textbf{12}(8), 1657--1664 (1995).
\newblock \doi{10.1364/JOSAA.12.001657}

\bibitem{Galati13}
Galati, L., Zheng, S.: Nonlinear {S}chr\"odinger equations for {Bose-Einstein}
  condensates.
\newblock AIP Conference Proceedings \textbf{1562}(1), 50--64 (2013).
\newblock \doi{10.1063/1.4828682}

\bibitem{gao16a}
Gao, Y., Mei, L.: Implicit--explicit multistep methods for general
  two-dimensional nonlinear {S}chr{\"o}dinger equations.
\newblock Appl. Numer. Math. \textbf{106}, 41--60 (2016).
\newblock \doi{10.1016/j.apnum.2016.06.003}

\bibitem{Gong14}
Gong, Y., Cai, J., Wang, Y.: Some new structure-preserving algorithms for
  general multi-symplectic formulations of {H}amiltonian \{PDEs\}.
\newblock Journal of Computational Physics \textbf{279}, 80 -- 102 (2014).
\newblock \doi{10.1016/j.jcp.2014.09.001}

\bibitem{Gong17}
Gong, Y., Wang, Q., Wang, Z.: Structure-preserving {G}alerkin {POD}
  reduced-order modeling of {H}amiltonian systems.
\newblock Computer Methods in Applied Mechanics and Engineering \textbf{315},
  780 -- 798 (2017).
\newblock \doi{10.1016/j.cma.2016.11.016}

\bibitem{Gong16}
Gong, Y., Wang, Y.: An energy-preserving wavelet collocation method for general
  multi-symplectic formulations of {H}amiltonian {PDE}s.
\newblock Communications in Computational Physics \textbf{20}(5), 1313--1339
  (2016).
\newblock \doi{10.4208/cicp.231014.110416a}

\bibitem{Hairer10}
Hairer, E., Lubich, C., Wanner, G.: Geometric numerical integration:
  Structure-preserving algorithms for ordinary differential equations.
\newblock Springer Series in Computational Mathematics. Springer, Heidelberg
  (2010).
\newblock \doi{10.1007/978-3-662-05018-7}

\bibitem{Halko11a}
Halko, N., Martinsson, P.G., Tropp, J.A.: Finding structure with randomness:
  Probabilistic algorithms for constructing approximate matrix decompositions.
\newblock SIAM Review \textbf{53}(2), 217--288 (2011).
\newblock \doi{10.1137/090771806}

\bibitem{Islas01}
Islas, A., Karpeev, D., Schober, C.: Geometric integrators for the nonlinear
  {Schr\"odinger} equation.
\newblock Journal of Computational Physics \textbf{173}(1), 116 -- 148 (2001).
\newblock \doi{10.1006/jcph.2001.6854}

\bibitem{Karasozen15nls}
Karas\"ozen, B., Akkoyunlu, C., Uzunca, M.: Model order reduction for nonlinear
  {S}chr{\"o}dinger equation.
\newblock Appl. Math. Comput. \textbf{258}, 509--519 (2015).
\newblock \doi{10.1016/j.amc.2015.02.001}

\bibitem{Karasozen13}
Karas{\"o}zen, B., \c{S}im\c{s}ek, G.: Energy preserving integration of
  bi-{Hamiltonian} partial differential equations.
\newblock Applied Mathematics Letters \textbf{26}(12), 1125 -- 1133 (2013).
\newblock \doi{10.1016/j.aml.2013.06.005}

\bibitem{Karasozen15}
Karas{\"o}zen, B., K{\"u}{\c{c}}{\"u}kseyhan, T., Uzunca, M.: Structure
  preserving integration and model order reduction of skew-gradient
  reaction--diffusion systems.
\newblock Annals of Operations Research \textbf{258}(1), 79--106 (2017).
\newblock \doi{10.1007/s10479-015-2063-6}

\bibitem{Karasozen17a}
Karas{\"o}zen, B., Uzunca, M., Sar{\i}ayd{\i}n-Fi̇li̇beli̇o\u{g}lu, A.,
  Y\"ucel, H.: Energy stable discontinuous {G}alerkin finite element method for
  the {Allen-Cahn} equation.
\newblock International Journal of Computational Methods \textbf{0}(0),
  1850,013 (0).
\newblock \doi{10.1142/S0219876218500135}

\bibitem{Koopman31}
Koopman, B.O.: Hamiltonian systems and transformation in {H}ilbert space.
\newblock Proceedings of the National Academy of Sciences \textbf{17}(5),
  315--318 (1931)

\bibitem{Kunisch01}
Kunisch, K., Volkwein, S.: {Galerkin} proper orthogonal decomposition methods
  for parabolic problems.
\newblock Numerische Mathematik \textbf{90}(1), 117--148 (2001).
\newblock \doi{10.1007/s002110100282}

\bibitem{Kutz16}
Kutz, J.N., Brunton, S.L., Brunton, B.W., Proctor, J.L.: Dynamic mode
  decomposition: Data-driven modeling of complex systems.
\newblock Society for Industrial and Applied Mathematics (SIAM), Philadelphia,
  PA (2016).
\newblock \doi{10.1137/1.9781611974508}

\bibitem{Lall03}
Lall, S., Krysl, P., Marsden, J.E.: Structure-preserving model reduction for
  mechanical systems.
\newblock Phys. D \textbf{184}(1-4), 304--318 (2003).
\newblock \doi{10.1016/S0167-2789(03)00227-6}

\bibitem{Li15}
Li, Y.W., Wu, X.: General local energy-preserving integrators for solving
  multi-symplectic {Hamiltonian} {PDEs}.
\newblock Journal of Computational Physics \textbf{301}, 141 -- 166 (2015).
\newblock \doi{10.1016/j.jcp.2015.08.023}

\bibitem{Mahoney11}
Mahoney, M.W.: Randomized algorithms for matrices and data.
\newblock Foundations and Trends in Machine Learning \textbf{3}(2), 123--224
  (2011).
\newblock \doi{10.1561/2200000035}

\bibitem{Martinsson16}
Martinsson, P.G.: Randomized methods for matrix computations and analysis of
  high dimensional data.
\newblock ArXiv e-prints  (2016)

\bibitem{Mezic13}
Mezi\'{c}, I.: Analysis of fluid flows via spectral properties of the {Koopman}
  operator.
\newblock Annual Review of Fluid Mechanics \textbf{45}(1), 357--378 (2013).
\newblock \doi{10.1146/annurev-fluid-011212-140652}

\bibitem{Mohebujjaman17}
Mohebujjaman, M., Rebholz, L.G., Xie, X., Iliescu, T.: Energy balance and mass
  conservation in reduced order models of fluid flows.
\newblock Journal of Computational Physics \textbf{346}(Supplement C), 262 --
  277 (2017).
\newblock \doi{10.1016/j.jcp.2017.06.019}

\bibitem{Peng16}
Peng, L., Mohseni, K.: Symplectic model reduction of {Hamiltonian} systems.
\newblock SIAM Journal on Scientific Computing \textbf{38}(1), A1--A27 (2016).
\newblock \doi{10.1137/140978922}

\bibitem{Pitaevskii03}
Pitaevskii, L.P., Stringari, S.: {B}ose-{E}instein condensation.
\newblock Clarendon Press, Oxford (2003)

\bibitem{Quispel08}
Quispel, G., McLaren, D.: A new class of energy-preserving numerical
  integration methods.
\newblock Journal of Physics A: Mathematical and Theoretical \textbf{41}(4),
  {045}{206} (7pp) (2008).
\newblock \doi{10.1088/1751-8113/41/4/045206}

\bibitem{riviere08dgm}
Riviere, B.: Discontinuous {Galerkin} Methods for Solving Elliptic and
  Parabolic Equations: Theory and Implementation.
\newblock SIAM (2008).
\newblock \doi{10.1137/1.9780898717440}

\bibitem{Rowley12}
Rowley, C.W., Mezi{\'c}, I., Bagheri, S., Schlatter, P., Henningson, D.S.:
  Spectral analysis of nonlinear flows.
\newblock Journal of Fluid Mechanics \textbf{641}, 115--127 (2009).
\newblock \doi{10.1017/S0022112009992059}

\bibitem{Schmid10dmd}
Schmid, P.J.: Dynamic mode decomposition of numerical and experimental data.
\newblock Journal of Fluid Mechanics \textbf{656}, 5--28 (2010).
\newblock \doi{10.1017/S0022112010001217}

\bibitem{Sulem99}
Sulem, C., Sulem, P.: The Nonlinear {S}chr{\"o}dinger Equation: Self-Focusing
  and Wave Collapse.
\newblock Applied Mathematical Sciences. Springer-Verlag New York (1999).
\newblock \doi{10.1007/b98958}

\bibitem{Tu14}
Tu, J.H., Rowley, C.W., Luchtenburg, D.M., Brunton, S.L., Kutz, J.N.: On
  dynamic mode decomposition: theory and applications.
\newblock J. Comput. Dyn. \textbf{1}(2), 391--421 (2014).
\newblock \doi{10.3934/jcd.2014.1.391}

\bibitem{Karasozen17}
Uzunca, M., Karas{\"o}zen, B.: Energy stable model order reduction for the
  {Allen-Cahn} equation.
\newblock In: P.~Benner, M.~Ohlberger, A.~Patera, G.~Rozza, K.~Urban (eds.)
  Model Reduction of Parametrized Systems, pp. 403--419. Springer International
  Publishing, Cham (2017).
\newblock \doi{10.1007/978-3-319-58786-8\_25}

\bibitem{Vemaganti07}
Vemaganti, K.: Discontinuous {G}alerkin methods for periodic boundary value
  problems.
\newblock Numer. Methods Partial Differ. Equ. \textbf{23}(3), 587--596 (2007).
\newblock \doi{10.1002/num.20191}

\bibitem{Wang13}
Wang, T., Guo, B., Xu, Q.: Fourth-order compact and energy conservative
  difference schemes for the nonlinear {S}chr{\"o}dinger equation in two
  dimensions.
\newblock Journal of Computational Physics \textbf{243}, 382 -- 399 (2013).
\newblock \doi{10.1016/j.jcp.2013.03.007}

\bibitem{Kutz13}
Williams, M.O., Schmid, P.J., Kutz, J.N.: Hybrid reduced-order integration with
  proper orthogonal decomposition and dynamic mode decomposition.
\newblock Multiscale Modeling \& Simulation \textbf{11}(2), 522--544 (2013).
\newblock \doi{10.1137/120874539}

\bibitem{Xu05}
Xu, Y., Shu, C.W.: Local discontinuous {Galerkin} methods for nonlinear
  {S}chr{\"o}dinger equations.
\newblock Journal of Computational Physics \textbf{205}(1), 72 -- 97 (2005).
\newblock \doi{10.1016/j.jcp.2004.11.001}

\bibitem{Xu12}
Xu, Y., Zhang, L.: Alternating direction implicit method for solving
  two-dimensional cubic nonlinear {S}chr{\"o}dinger equation.
\newblock Computer Physics Communications \textbf{183}(5), 1082--1093 (2012).
\newblock \doi{10.1016/j.cpc.2012.01.006}

\bibitem{Zimmermann16}
Zimmermann, R., Willcox, K.: An accelerated greedy missing point estimation
  procedure.
\newblock SIAM Journal on on Scientific Computing \textbf{38}(5), A2827--A285
  (2016).
\newblock \doi{10.1137/15M1042899}

\end{thebibliography}
\end{document}